\documentclass[12pt]{amsart}

\usepackage{mathtools}
\mathtoolsset{showonlyrefs=true}
\usepackage[hmargin=0.8in,height=8.6in]{geometry}
\usepackage{amssymb,amsthm, times}
\usepackage{delarray,verbatim}
\usepackage{ifpdf}
\ifpdf
\usepackage[pdftex]{graphicx}
 \else
\usepackage[dvips]{graphicx}
 \fi

\usepackage{bm}

\usepackage{amsmath}

\usepackage{ifpdf}
\usepackage{color}
\definecolor{webgreen}{rgb}{0,.5,0}
\definecolor{webbrown}{rgb}{.8,0,0}
\definecolor{emphcolor}{rgb}{0.5,0.95,0.95}

\usepackage{hyperref}
\hypersetup{%
	colorlinks=true,
	linkcolor=webbrown,
	filecolor=webbrown,
	citecolor=webgreen,
	breaklinks=true}
\ifpdf \hypersetup{pdftex,
	pdfstartview=FitH, 
	bookmarksopen=true,
	bookmarksnumbered=true
} \else \hypersetup{dvips} \fi
\allowdisplaybreaks

\newcommand {\bP}{\mathbf{P}}
\newcommand {\bbE}{\mathbf{E}}

\newcommand {\B}{\mathcal{B}}
\newcommand {\bb}{\mathfrak{b}}
\newcommand {\qq}{\mathfrak{q}}

\numberwithin{equation}{section}

\newtheorem{theorem}{Theorem}[section]
\newtheorem{proposition}{Proposition}[section]
\newtheorem{corollary}{Corollary}[section]
\newtheorem{remark}{Remark}[section]
\newtheorem{lemma}{Lemma}[section]

\newtheorem{assump}{Assumption}[section]

\numberwithin{remark}{section} \numberwithin{proposition}{section}
\numberwithin{corollary}{section}
\newcommand {\R}{\mathbb{R}}

\newcommand {\cF}{\mathcal{F}}

\newcommand {\cA}{\mathcal{A}}
\newcommand {\N}{\mathbb{N}}
\newcommand {\p}{\mathbb{P}}

\newcommand {\E}{\mathbb{E}}
\newcommand {\bE}{\mathbb{E}}

\newcommand {\bR}{\mathbb{R}}

\newcommand{\diff}{{\rm d}}

\newcommand{\lev}{L\'{e}vy }

\newcommand{\Ind}{\mathbb{I}}

\newcommand{\DD}{\mathcal{D}}

\begin{document}
	\title[Optimal dividends and capital injection]{Optimal dividends and capital injection:
	 A general L\'evy model with extensions to regime-switching models}

	\thanks{This version: \today.   }
	
	\author[D. Mata L\'opez]{Dante Mata L\'opez$^\sharp$}
		\thanks{$\sharp$\, D\'epartement de math\'ematiques, Universit\'e de Qu\'ebec \`a Montr\'eal (UQAM), 201 av. Pr\'esident-Kennedy, Montr\'eal H2X 3Y7, QC, Canada. Email: mata\_lopez.dante@uqam.ca}
	\author[K. Noba]{Kei Noba$^\dagger$}
\thanks{$\dagger$\, Department of Fundamental Statistical Mathematics, The Institute of Statistical Mathematics, 
10-3 Midori-cho, Tachikawa-shi, Tokyo 190-8562, Japan. Email: knoba@ism.ac.jp}
	\author[J. L. P\'erez]{Jos\'e-Luis P\'erez$^*$}
	\thanks{$*$\, Department of Probability and Statistics, Centro de Investigaci\'on en Matem\'aticas A.C. Calle Jalisco
		s/n. C.P. 36240, Guanajuato, Mexico. Email: jluis.garmendia@cimat.mx}
	\author[K. Yamazaki]{Kazutoshi Yamazaki$^\ddagger$}
\thanks{$\ddagger$\, School of 
Mathematics and Physics, The University of Queensland, St Lucia,
Brisbane, QLD 4072, Australia. Email: k.yamazaki@uq.edu.au}
	\date{}
	
	\begin{abstract}
		This paper studies a general \lev process model of the bail-out optimal dividend problem with an exponential time horizon,
 and further extends it to the regime-switching model.  We first show the optimality of a double barrier strategy in  the single-regime setting with a concave terminal payoff function. This is then applied to show the optimality of a Markov-modulated
  double barrier strategy in the regime-switching model via contraction mapping arguments.
 We solve these for a general \lev model with both positive and negative jumps, greatly generalizing the existing results on spectrally one-sided models.
		\\
		\noindent \small{\noindent  MSC 2020: 60G51
93E20
91G80
} 
	\end{abstract}
	
	\maketitle
	
	\section{Introduction}
	
	The bail-out optimal dividend problem is one important extension of de Finetti's optimal dividend problem.  Given a surplus process of a company, the objective is to derive an optimal combination of dividend and capital injection strategies  that maximizes the expected \textit{net present value} (NPV) of dividend payments minus capital injections. Unlike the classical formulation without bail-out, a strategy must be selected so that the surplus process must be kept non-negative  to avoid bankruptcy. A majority of the existing results in this subject focus on the analysis of the so-called \emph{double-barrier} strategy. It pays dividends whenever the surplus attempts to go above a certain upper barrier, say $b^*$, and injects capital when it attempts to go below zero, so that the resulting surplus process stays on the interval 
$[0, b^*]$. This is indeed a natural candidate of optimal strategy and a number of existing results have successfully verified this conjecture \cite{AvrPalPis2007,Perez_yamazaki_yu,WWCW,ZCY, ZWYC}.

The inclusion of a random termination time is another crucial extension of de Finetti's problem, which is important on its own and also provides a direct connection with a generalization with regime-switching.  It offers a more general model than the classical infinite horizon problem and supports a wide array of applications by suitably selecting the terminal payoff or cost function. This versatility is exemplified by the regime-switching model considered in this paper.
 In the regime-switching model, also known as the Markov additive (or Markov modulated) model \cite{JP2012,MMNP, NobPerYu,LL, ZY}, the dynamics of the surplus process changes depending on the macroeconomic conditions modeled in terms of a continuous-time Markov chain. The optimal strategy needs to be modulated according to the underlying state; it is therefore of importance to pursue the optimality of a Markov-modulated version of the double-barrier strategy.
To tackle this, robust results/analysis on a single-regime model with an exponential termination time are particularly important. This is because the regime-switching model, where the transition of the Markov chain occurs at (state-dependent) exponential times, can be written in terms of a collection of multiple single-regime models with exponential termination times. 
By taking advantage of this and contraction mapping arguments, Jiang and Pistorius \cite{JP2012} successfully showed the existence of a barrier-type strategy that is optimal, focusing on the Brownian motion model in the classical de Finetti's problem. Recently, it has been extended to more general spectrally one-sided \lev models by \cite{MMNP, NobPerYu}.
%
%
%
%

	\par
	Despite this progress of the theory related to the bail-out optimal dividend problem and also its extensions, a majority of the existing results assume  spectrally one-sided \lev processes (i.e.\ \lev processes with only negative or positive jumps 
	as their underlying processes. 
	The spectrally one-sided model in de Finetti's optimal dividend problem has been studied extensively in the last couple of decades, mainly due to the existence of the so-called \emph{scale function} (see \cite{Ber1996, Kyp2014}). This is because it allows us to express -- very concisely -- the NPV of dividends and capital injections for a general spectrally one-sided \lev process and hence to solve the optimal bail-out dividend problem directly.

	
%
	
	
	
	In this paper, we study  a general \lev process model, allowing both positive and negative jumps, for which the scale-function approach is no longer available. We consider the case with an exponential termination time, so that it can be applied to solve the regime-switching extension.
	
Employing a general Lévy process allows for a flexible model of a risk process over classical Cramér-Lundberg and spectrally negative L\'evy processes. Positive jumps arise when considering random gains (see, e.g., the introduction of \cite{Albrecher, MGK}). The importance of the two-sided jump model is evidenced by several papers in the literature, where the authors use Lévy or related stochastic processes with two-sided jumps to model risk processes \cite{Cheung, PV, WXY, ZhangYang, ZYL}. In particular, Martín-González et al. \cite{MGMSP} studied the Gerber-Shiu function for Markov-modulated Lévy risk processes with two-sided jumps. Zhang \cite{Zhang} studied the Erlang(n) risk model with arbitrary negative jumps and positive exponential jumps with the existence of a constant dividend barrier. Employing a general Lévy model also offers a flexible approach to modeling perturbations, both Gaussian and non-Gaussian. By considering processes with (possibly an infinite number of) small positive jumps, it allows for more realistic models, alternatives to classical Brownian and spectrally negative stable perturbation \cite{Kolkovska}.

The case for a general \lev model is significantly more challenging than the spectrally one-sided case and thus established results on the bail-out dividend problem for a general \lev process are extremely limited. This is due to the fact that unlike the case with one-sided jumps, in the general case one cannot rely on scale functions to show the optimality of a double-barrier strategy.
However, the recent works by \cite{Nob2019} and \cite{NobYam2020} have provided us with a framework to deal with control problems driven by a general \lev process (with minor conditions), without relying on the scale function. In particular, Noba \cite{Nob2019} solved the infinite-time-horizon single-regime bail-out optimal dividend problem  for a general \lev process and showed the optimality of a double-barrier strategy.


	
	
	
	Our methodologies and contributions (in particular beyond  \cite{Nob2019})
	are summarized as follows:
	\begin{enumerate}
	\item We first consider the (single-regime) problem with an independent exponential time horizon driven by a general \lev process with two-sided jumps.
			In order to prove the optimality of a double barrier strategy,  
			we first obtain a probabilistic expression for the derivatives of the NPV associated with a double reflection strategy, to analyze the concavity and smoothness of the associated (candidate) value function. This is done by performing an analysis of how the sample paths of the controlled \lev process change when we add a small perturbation to the initial value or the value of the reflection barrier. Although this analysis follows a similar line of reasoning as in \cite{Nob2019}, our problem requires extensive techniques not provided in \cite{Nob2019} because we also need to analyze the terminal payoff term.
			
Despite the fact that the scale function is not available in our problem, our expression of the derivative of the NPV allows us to follow the \textit{guess and verify} procedure in order to find a suitable candidate for the optimal barrier strategy. It is chosen by using the conjecture that the slope of the value function at the upper reflection barrier becomes one. For this, it is essential to incorporate the effect of the termination time and the terminal payoff. 
The optimality of our candidate double barrier strategy is then verified rigorously by 		 
showing that the candidate value function satisfies the proper variational inequalities.

	\item We then show how our problem with exponential termination is applied in order to deal with the more complex bailout dividend problem driven by a regime-switching \lev process, also known as \textit{Markov Additive Process} (MAP) (see, e.g., \cite{Asmussen, Ivanovs_Palmowski}), with two-sided jumps. This analysis relies on using the dynamic programming principle in order to reduce the dimension of the general problem into a collection of bailout dividend problems with exponential termination driven by a single \lev process, where we have already proven that a double reflection strategy is optimal. 
	In addition, we use the dynamic programming equation in order to define a recursive operator, and prove
	that both the value function and the NPV associated with our candidate barrier strategy are solutions of a functional equation associated with said recursive operator. 
	This allows us to use iterative arguments to show that the expected NPV associated with our candidate optimal strategy agrees with the value function. 
	\end{enumerate}
	The rest of this paper is structured as follows. In Section \ref{Sec_Bailout}, we provide the setting of the bailout dividend problem with exponential termination, and driven by a \lev process with two-sided jumps. We also introduce the set of double barrier strategies, and we also propose and select our candidate optimal strategy. In Section \ref{section_verification}, we provide a verification lemma and we perform a rigorous verification of optimality of our candidate barrier strategy, by showing that its associated NPV solves the proper variational inequalities and satisfies the conditions of the verification lemma. Finally, in Section \ref{sec3}, we apply the obtained results to solve the regime-switching case,
	driven by a MAP with two-sided jumps. 
	We introduce the definition of a MAP, as well as the class of regime-modulated double barrier strategies. We then conclude the section by proving the existence and optimality of a regime-modulated double barrier strategy via iterative operators together with numerical examples. The proofs of some technical lemmas can be found in Appendix \ref{section_proofs}. 
For the rest of paper, we will distinguish the two notations $f'(x)$ and $f'_+(x)$ for a given function $f(x)$; the former denotes its standard derivative and the latter  denotes its right derivative. 
In addition, almost everywhere (a.e.) is understood in the sense of Lebesgue measure, unless stated otherwise.
	
	\section{Bail-out dividend problem with termination at an exponential time}\label{Sec_Bailout}
	We will introduce and study a modification of the bail-out optimal dividend problem driven by a general \lev process, where we add a final payoff (or cost) at an independent exponential terminal time. Our main result for a general L\'evy model is new to the literature.
	\subsection{Preliminaries}\label{Sec02}
	{Let $(\Omega , \mathcal{F}, \p)$  be} a probability space hosting a {one-dimensional} \lev process $X=\{{X(t)} : t \geq 0 \}$.  We denote by $\mathbb{F} := (\mathcal{F}(t): t \geq 0)$ the filtration generated by $X$. 
	For $x\in\R$, we denote by $\p_x$ the law of $X$ when it starts at $x$ and by $\E_x$ its corresponding expectation operator. In particular, we {use the notation} $\mathbb{P} = \mathbb{P}_0$ and $\E = \E_0$. Throughout the paper, let $\psi$ be the characteristic exponent of $X$ that satisfies 
	\begin{align}
		e^{-t\psi(\lambda)}=\E\left[e^{\mathbf{i}\lambda {X(t)}}\right],  \quad \lambda \in \R, ~t\geq 0. 
	\end{align} 
	The characteristic exponent $\psi$ is known to take the form  
	\begin{align}
		\psi (\lambda) := -\mathbf{i}\gamma\lambda +\frac{1}{2}\sigma^2 \lambda^2 
		+\int_{\R \backslash \{0\} } (1-e^{\mathbf{i}\lambda x}+\mathbf{i}\lambda x \Ind_{\{|x|<1\}}) {\nu}(\diff x) , ~~~~~~\lambda\in\R. \label{202a}
	\end{align}
	Here, $\gamma\in\R$, $\sigma\geq 0$, and $\nu $ is a L\'evy measure on $\R \backslash \{0\}$ such that 
	\begin{align}
		\int_{\R\backslash \{ 0\}}(1\land x^2)  {\nu} (\diff x) < \infty. \label{56}
	\end{align}
	Recall that the process $X$ has bounded variation paths if and only if $\sigma= 0$ and $\int_{|x|<1} |x|\nu(\diff x)<\infty$. When this holds, we can write 
	\begin{align}
		\psi(\lambda) := -\mathbf{i}\delta+\int_{\R \backslash \{0\} }  (1-e^{\mathbf{i} \lambda x}) {\nu} (\diff x),  \label{204}
	\end{align}
	where
	\begin{align}
		\delta := \gamma-\int_{|x|<1}x  {\nu}(\diff x). \label{def_delta}
	\end{align}



	
	\subsection{{Bail-out optimal dividend problem with exponential terminal time}.} \label{subsection_single_regime_formulation}
	
	A dividend/capital injection strategy $\pi = \{ ({L^\pi(t), R^\pi(t)}) : t \geq 0 \}$ is {a pair of} non-decreasing, right-continuous and $\mathbb{F}$-adapted processes  such that $L^\pi(0-) = R^\pi(0-) = 0$,
	where $L^\pi$ models the cumulative amount of dividends and $R^\pi$ models that of injected capital. 
	
	Given a strategy $\pi$, its corresponding risk process is given by ${U^\pi(0-)}:=x$ under $\p_x$ and
	\begin{align*}
		{U^\pi(t)}:= {X(t)}-{L^\pi(t)}+{R^\pi(t)}, ~~~t\geq 0.
	\end{align*}
	It is required that almost surely
	\begin{align}\label{positiveness}
	{U^\pi(t)} \geq 0 \quad \textrm{ for all } t \geq 0.
	\end{align}
	  
	{Let $\zeta$ be an independent exponential random variable with parameter $r> 0$ {(mean $1/r$)}
	  {independent of $X$}.} 
	  {It models}
	   a random termination time, {upon which}
	    a final payoff {(or cost)} 
	     $w(U^{\pi}(\zeta))$ {is collected. The function}
	      $w(\cdot):[0, \infty)\to\R$ is assumed to satisfy Assumption \ref{assum_w} below. 	
	      It is worth noting that the independence between $X$ and $\zeta$ implies that $X$ does not jump at $\zeta$ and hence ${U^\pi(\zeta -)= U^\pi(\zeta)}$ a.s.
	
	{Let} $\beta> 1$ {be} the cost per unit of injected capital and $q>0$ {be} the discount factor.  {Associated with a strategy $\pi$ with initial capital $x\geq 0$, the expected net present value (NPV) of the dividend payments and the terminal payoff/cost minus the cost of capital injection  is defined by}
	\begin{align}
		v_\pi (x) &:= \E_{x} \left[\int_{[0, \zeta)} e^{-qt}\diff {L^\pi(t)}  -\beta\int_{[0, \zeta)} e^{-qt} \diff {R^\pi(t)}
		+e^{-q\zeta} w({U^\pi(\zeta)})\right]\\
		&=\E_{x} \left[\int_{[0, \infty)} e^{-\alpha t}\diff {L^\pi(t)} -\beta\int_{[0, \infty)} e^{-\alpha t} \diff {R^\pi(t)}
		+r\int_0^\infty e^{-\alpha t} w({U^\pi(t)}) \diff t
		\right],
	\end{align}
		where 
	\begin{align} \label{alpha_def}
		\alpha := q + r. 
	\end{align} 
For the {terminal} payoff function $w$, we also require the following conditions. {These are slightly weaker than those assumed in the literature, in particular,}
 \cite[Assumption 4.3]{NobPerYu} and \cite[Assumption 4.2]{MMNP}. 
	\begin{assump}\label{assum_w}\label{AA3}
		It is assumed that $w$ is non-decreasing, continuous and concave on $[0,\infty)$. {In addition, } $w'_+(0+) \leq \beta\alpha/r$
		and $w'_+(\infty):=\lim_{x\to\infty}w'_+(x) \in[0,\alpha/r)$, 
		where we use $w'_+(x)$ to represent the right derivative of the concave function $w(x)$. 
	\end{assump}

	\par 	
Assumption \ref{assum_w} ensures the optimality of a barrier strategy (see Theorem \ref{Thm401}). Importantly, this is satisfied in the application of the regime-switching case we describe in Section \ref{sec3}.



 For the expectation $v_\pi(x)$ to be well-defined, it is required that
		\begin{align}
			\E_{x} \left[ \int_{[0, \zeta)} e^{-qt} \diff {R^\pi(t)}\right] {=\E_{x} \left[ \int_{[0,\infty)} e^{- \alpha t} \diff {R^\pi(t)}\right]}< \infty, \label{9}
		\end{align}
			Note that $\E_x \left[e^{-q\zeta} w (U^\pi(\zeta) ) \right] \geq w(0)>-\infty$ by Assumption \ref{assum_w}.
	
	We call a strategy $\pi$, an \emph{admissible strategy}, if it satisfies both the positivity assumption \eqref{positiveness}
	as well as the integrability condition \eqref{9}.
	
	The value function for the stochastic control problem is given by
	\begin{align}
	{v} (x) := \sup_{\pi \in \Pi} v_\pi (x),\qquad\text{$x\in\R$,}\label{vf_def}
	\end{align}
	{where $\Pi$ denotes the set of all admissible strategies. }
	
	
	We aim in characterizing the value function of the problem and obtaining the optimal dividend-capital injection strategy $\pi^*$ which achieves the value function in \eqref{vf_def}, {if it exists}.
	
	Throughout this and next sections, we will make the next standard assumption on the L\'evy process $X$.
	\begin{assump}\label{AA1}
		We assume that $\E\left[|{X(1)}|\right]<\infty$. This is equivalent to the condition 
		$\int_{\bR \backslash [-1, 1]} |x| \nu (\diff x)< \infty$
		by \cite[Theorem 3.8]{Kyp2014}. 
	\end{assump}

\begin{assump} \label{assumption_compound_poisson}
We assume that $X$ is not a driftless compound Poisson process. In other words, we assume that $X$ has bounded variation paths and satisfies $\delta \neq 0$ {(see \eqref{def_delta})} or $\nu(\R\backslash\{0\})=\infty$, or  has unbounded variation paths.
\end{assump}
\begin{remark} In this paper, we focus on the setting where the ruin must be avoided as in many papers in the literature (e.g., \cite{AvrPalPis2007}.). However, it is also of interest to consider the case when ruin is also allowed, with
	\begin{align*}
		v_\pi (x) := \mathbb{E}_{x} \left[\int_{[0, \eta_0^-)} e^{-qt} \diff {L^\pi(t)}  -\beta\int_{[0,\eta_0^-)} e^{-qt} \diff  {R^\pi(t)}
		+e^{-q\tau_0^-} g({U^\pi(\eta_0^-)})\right]
	\end{align*}
	where $g$ models a reward/penalty at ruin $\eta_0^-$ at which the surplus process goes strictly below zero. 
\end{remark}

	\subsection{Double barrier strategies for L\'evy processes}\label{Sec103}
We introduce our candidate optimal strategy, which {we call} a double reflection strategy ${\pi^{(0,b)}} =\{  ({L^{(0,b)}(t), R^{(0,b)}(t))}  :  t \geq 0\}$ with an {upper} reflection barrier $b\geq0$ {and a lower reflection barrier $0$}.
	Under this strategy, 
	the process is pushed downward by paying dividends whenever the process attempts to upcross above $b$, while the process is pushed upward by injecting capital whenever the process attempts to downcross below $0$. 
	The resulting surplus process
	\[
	{U^{(0,b)}(t)=X(t)-L^{(0,b)}(t)+R^{(0,b)}(t)},\qquad\text{$t \geq0$,}
	\]
	is the so-called doubly reflected L\'evy process (for the definition, see, e.g., \cite[Section 4]{AvrPalPis2007} or \cite[Section 3]{Nob2019}).
	
\begin{remark} \label{remark_zero_zero} 
With $b=0$, the strategy $\pi^{(0,0)}$  controls the process to stay at $0$ uniformly in time. This makes sense only when $X$ is of bounded variation. Otherwise, it is not achievable (violating \eqref{9}). Hence we only consider $\pi^{(0,0)}$ for the case of bounded variation.
\end{remark}	

	
	For $b \geq 0$, we denote the corresponding expected NPV by
	\begin{align}\label{v_pi}
		v_{b} (x) &
		:=v_{\pi^{(0,b)}}(x) 
		= v_b^L(x) - \beta v_b^R(x) + r v_b^w(x), \quad x \in \R,
	\end{align}
	where
	\begin{align}
		v_b^L (x)&:=\E_{x} \left[\int_{[0, \infty)} e^{-\alpha t}\diff  {L^{(0,b)}(t)} \right], \qquad
		v_b^R(x) :=\E_x \left[ \int_{[0, \infty)} e^{-\alpha t} \diff {R^{(0,b)}(t)}\right], \notag\\
		v_b^w(x)&:=\E_x\left[\int_0^\infty e^{-\alpha t} w( {U^{(0,b)}(t)})\diff t
		\right]. 
	\end{align}
	This decomposition makes sense by the following lemma, for $b \geq 0$ in the bounded variation case and $b > 0$ in the unbounded variation case {(recall Remark \ref{remark_zero_zero})}. This decomposition will simplify the subsequent analysis of the problem.  
	
	
	\begin{lemma}\label{Lem301}
		For $b >0$, {$v_b^L(x)<\infty$, $v_b^R(x)<\infty$ and $|v_b^w(x)|<\infty$} 
		and hence the double barrier strategy  {$\pi^{(0,b)}$} is admissible. 
		{{When $b = 0$ and  $X$ has paths of bounded variation,} then  {$\pi^{(0,0)}$}} is admissible.  
	\end{lemma}
	\begin{proof}
		From Assumption \ref{AA1} {together with} \cite[Lemma 3.2]{Nob2019}, we obtain {that} $v_b^L(x)<\infty$ and $v_b^R(x)<\infty$ for $b >0$ and $x \in \R$. {Additionally, using \cite[Lemma 3.3]{Nob2019} we obtain that this fact is true for $b =0$ and $x \in \R$ when $X$ {has paths of bounded variation}.} {Because $w$ is locally bounded and $0 \leq U^{(0,b)}(t) \leq b$ for all $t \geq 0$ a.s., we also have $ {|v_b^w(x)|} < \infty$ for all $x \in \R$.} 
	\end{proof}

{Our main objective is to show  the optimality of a double-barrier strategy with a suitable (upper) barrier $b^*$. We will show this in two steps. We first show its optimality over the set of barrier strategies 
	i.e.\ ${\tilde{\Pi} :=} (\pi^{(0,b)}; b \geq 0) {\subset \Pi}$ (resp.\ ${\tilde{\Pi} :=} ( \pi^{(0,b)}; b > 0){\subset \Pi}$) when $X$ is of bounded (resp.\ unbounded variation) in Lemma \ref{lemma_g_monotonicity}. This version of optimality will be strengthened 
to the optimality over all admissible strategies $\Pi$ in Section \ref{section_verification}.
}
	{
	


	\subsection{Selection of the candidate threshold}\label{Subsec_Aux_Candidate_Threshold}
	
	
	{First, we define our candidate barrier, which we call $b^*$,  in terms of a {(single-sided)}  reflected L\'evy process $Y^b=\{{Y^{b}(t)}: t\geq0\}$ with an upper barrier $b$,}
		 with
	\begin{align}
		{Y^{b}(t)}={X(t)}-\Big\{\sup_{s\in[0,t]}({X(s)}-b) \lor 0\Big\}, \quad t\geq 0 , \label{A001}
	\end{align}
	and also its first passage time 
		\begin{align}
		\kappa^{b, -}_a :=\inf\{t\geq 0 : {Y^{b}(t)} < a\}, \quad a \in \R. \label{kappa_def}
		\end{align}
	
	This process $Y^b$ and the first passage time $\kappa^{b,-}_a$ play important roles in concisely characterizing the optimal strategy: we set our candidate upper barrier as {the inverse}
	\begin{align}
		b^\ast:= \inf\left\{b\geq 0 : g (b) < 1 \right\}, \label{threshold}
	\end{align}
	{of the function}
	\begin{align} \label{def_g}
	\begin{split}
		g(b)&:= \beta\mathbb{E}_b\left[e^{-{{\alpha }}\kappa^{b,-}_0} \right]+r\mathbb{E}_b\left[\int_0^{\kappa^{b,-}_0}e^{-\alpha t}w'_+( {Y^b(t)}) \diff t\right]\\
		&=\beta - \E_b \left[ \int_0^{\kappa^{b,-}_0} e^{-\alpha t} l ( {Y^b(t)}) \diff t \right] = \beta - \E_0 \left[ \int_0^{\kappa^{0,-}_{-b}} e^{-\alpha t} l ( {Y^0(t)}+b) \diff t  \right], \quad b \geq 0.
		\end{split}
	\end{align}
	{
Here, we define
	\[
	l(y) := \beta \alpha -rw'_+ (y), \quad y \geq 0,
	\]
	which satisfies the properties summarized in the following remark.
	}

\begin{remark} \label{remark_l} 
By  Assumption \ref{assum_w}, the function $l$ is non-decreasing with $l(0) = \beta \alpha - r w_+'(0)  \geq 0$ and $l(\infty) = \beta \alpha - r w_+'(\infty) > (\beta-1) \alpha$. 
\end{remark}

	\begin{lemma} \label{lemma_g_monotonicity} 
	The function $g$ is non-increasing and 
		\[
	\lim_{b \to \infty}g(b) = \frac  r \alpha w'_+(\infty) <1.
	\]
	
	\end{lemma}
	\begin{proof} 
	
	By Remark \ref{remark_l},  the mapping $b \mapsto l (Y^0(t)+b)$ is non-negative and non-decreasing. In addition, $b \mapsto \kappa^{0,-}_{-b}$ is non-decreasing. This shows in view of the last expression of \eqref{def_g} that $g$ is non-increasing.
	
For the second claim, dominated convergence, together with $\kappa^{0,-}_{-b} \xrightarrow{b \uparrow \infty} \infty$, gives
$\lim_{b \to \infty} g(b) = \beta - \E_0 \left[ \int_0^\infty e^{-\alpha t} l ( \infty) \diff t  \right] = \frac  r \alpha w'_+(\infty)$ (see Remark \ref{remark_l}),
which is {less than $1$} by  Assumption \ref{AA3}.

%
	\end{proof}
	
	By Lemma \ref{lemma_g_monotonicity}, {our candidate barrier} $b^*$ is well-defined and finite. On the other hand, it may take the value $0$ depending on the path variation of $X$, as remarked below.
	
	
\begin{remark} \label{Rem202}
If $b^* = 0$, then
 $g(0) \leq 1$ 
where
	\begin{align} 
		g(0)= \beta - \E_0 \left[ \int_0^{\kappa^{0,-}_{0}} e^{-\alpha t} l( 0) \diff t  \right] = \beta - \left(\beta - \frac {rw'_+ ( 0)}  \alpha \right) \Big(1- \E_0 \left[ e^{-\alpha\kappa^{0,-}_{0}} \right] \Big).
	\end{align}
	\begin{enumerate}
\item When $X$ is of unbounded variation {or of bounded variation with negative drift {or $\nu (-\infty , 0)=\infty$}}, we have $\kappa^{0,-}_{0} = 0$ 
a.s.\ and hence $g(0) = \beta > 1$, showing $b^* > 0$.
\item {Otherwise} (i.e.\ $X$ is of bounded variation with {non-negative} drift {and $\nu(-\infty,0)<\infty$}), $\kappa^{0,-}_0>0$ a.s.\ and $b^\ast$ may take value $0$. 

{In this case, $\kappa^{0,-}_0$ is the first downward jump time of $X$, which is exponentially distributed with parameter $\nu(-\infty,0)$, and thus we can write
$$
g(0)=\beta - \left(\beta - \frac {rw'_+ ( 0)}  \alpha \right) \Big(1- \frac{\nu(-\infty,0)}{\alpha +\nu(-\infty,0)} \Big). 
$$
}
Therefore, if $b^*=0$ then
\begin{align}
\nu(-\infty,0) (\beta - 1)-\alpha +rw^{\prime}_+(0) \leq 0. \label{52}
\end{align}

%
\end{enumerate}
Note that, $g(0) \xrightarrow{{\alpha \uparrow \infty}} 0 < 1$ and hence $b^* = 0$  for large enough $r$ or $q$ (recall \eqref{alpha_def}).
%
\end{remark}	


	In the next result we will show that the function $g$ is right-continuous. The proof is given in Appendix \ref{proof_Lem402}.
	
	\begin{lemma}\label{Lem402} 
		The {function $g$} is right-continuous {on $\R$}. In addition, if the point $0$ is regular for $(-\infty,0)$ (i.e.\ $\p(T_{(-\infty, 0)}=0)=1$, where $T_{(-\infty, 0)}:= \inf \{t>0 : X(t) \in (-\infty, 0) \}$),  then it is also left-continuous {(and thus continuous)}. 
	\end{lemma}

{We conclude this section by confirming its optimality  over all double reflection strategies $\tilde{\Pi}$.  This will be used to show the optimality over all admissible strategies $\Pi$ (see the proof of Lemma \ref{Lem410}).}

	\begin{lemma} \label{lemma_opt_barrier}
		For $b>0$ (resp., $b\geq 0$), we have 
		\begin{align} \label{v_b_max}
			v_{b^\ast}(x) \geq v_b (x), \quad x\in\R, 
		\end{align}
		when $X$ has {unbounded} (resp., {bounded}) variation paths. 
	\end{lemma}
	\begin{proof}
	{To show this lemma, we compute the derivative of $v_b(x)$ with respect to $b$ and verify that $b^*$ is the maximizer.}
	
		{We first suppose} $b>0$. 
		We define {a} sequence of hitting times 
		\[
		{0 =:  \rho^{b, (0)}_0 < \rho^{b, (1)}_b <  \rho^{b, (1)}_0 <  \cdots < \rho^{b, (n)}_b <  \rho^{b, (n)}_0 <\cdots}
		\]
		recursively defined by
		\begin{align*}
		\rho^{b, (n)}_0:= \inf\{t>\rho^{b, (n)}_b:  {R^{(0,b)}(t)}> {R^{(0,b)}(\rho^{b, (n)}_b)}\}, \quad
			\rho^{b, (n)}_b:= \inf\{t> \rho^{b, (n-1)}_0:  {L^{(0,b)}(t)}> {L^{(0,b)}(\rho^{b, (n-1)}_0)}\},
		\end{align*}
		for $n\in\N$. In particular, $\rho^{b, (1)}_b:= \inf\{t> 0:  {L^{(0,b)}(t)}> 0 \}$ is the first upper reflection time above. {Note that, because reflection from above (resp.\ below) happens at $\rho^{b, (n)}_b$ (resp.\ $\rho^{b, (n)}_0$), we must have
		\begin{align}
		U^{(0,b)}(\rho^{b, (n)}_{b}) = b \quad \textrm{and} \quad U^{(0,b)}(\rho^{b, (n)}_{0}) = 0, \quad b > 0, n \geq 1.  \label{rho_b_jump}
		\end{align}
		}
		\par
		For $b>0$ and $x\in\R$, we have {by \cite[Lemma 4.2]{Nob2019}}
		\begin{align}
			\lim_{\varepsilon\downarrow0}\frac{v^L_{b+\varepsilon}(x)-v^L_{b}(x)}{\varepsilon}=&\frac{-\E_x \left[e^{-\alpha\rho^{b, (1)}_b}\right]}{1-\E_b\left[e^{-\alpha\kappa^{b,-}_0}\right]\E_0 \left[e^{-\alpha\rho^{b, (1)}_b}\right]}, \label{12}\\
			\lim_{\varepsilon\downarrow0}\frac{v^R_{b+\varepsilon}(x)-v^R_{b}(x)}{\varepsilon}=&\frac{-\E_x \left[e^{-\alpha\rho^{b, (1)}_b}\right]\E_b\left[e^{-\alpha\kappa^{b,-}_0}\right]}{1-\E_b\left[e^{-\alpha\kappa^{b,-}_0}\right]\E_0 \left[e^{-\alpha\rho^{b, (1)}_b}\right]}. \label{13}
		\end{align}
		\par
		We fix $\varepsilon>0$. 
		From the equation following identity (4.7) in \cite{Nob2019}, we know 
		\begin{align}
			0\leq  {U^{(0,b+\varepsilon)}(t)-U^{(0,b)}(t)} \leq \varepsilon, \quad t\geq 0.  \label{6}
		\end{align}
		We define, for $n\in\N$, 
		\begin{align}
		\bar{\rho}^{b, (n)}_b :&= \inf\{t >\rho^{b+\varepsilon, (n-1)}_{0}: { L^{(0,b)}(t) >L^{(0,b)}(\rho^{b+\varepsilon, (n-1)}_{0})} \}, \notag\\
			\bar{\rho}^{b, (n)}_0 :&= \inf\{t > \rho^{b+\varepsilon, (n)}_{b+\varepsilon}:  {R^{(0,b)}(t)>R^{(0,b)}(\rho^{b+\varepsilon, (n)}_{b+\varepsilon})} \}.  
		\end{align} 
		
		{For $n \in \N$,} by \eqref{6} and since $ {U^{(0,b+\varepsilon)}(\rho^{b+\varepsilon, (n)}_{b+\varepsilon})}=b+\varepsilon$ {(see \eqref{rho_b_jump})} {and $U^{(0,b)} \leq b$ uniformly in time}, 
		\begin{align}
			 {U^{(0,b+\varepsilon)}(\rho^{b+\varepsilon, (n)}_{b+\varepsilon})-U^{(0,b)}(\rho^{b+\varepsilon, (n)}_{b+\varepsilon})}=\varepsilon. \label{7}
		\end{align}
		{On $[{\rho^{b+\varepsilon, (n)}_{b+\varepsilon}}, \bar{\rho}^{b, (n)}_0)$, before $U^{(0,b)}$ (and hence $U^{(0, b + \varepsilon)}$ as well) is  reflected from below, }
		by  \eqref{7} and the definitions of  {$U^{(0,b+\varepsilon)}$} and  {$U^{(0,b)}$}, the difference stays the same, i.e., 
		\begin{align}
			 {U^{(0,b+\varepsilon)}(t) -U^{(0,b)}(t)} =\varepsilon, \quad t\in [{\rho^{b+\varepsilon, (n)}_{b+\varepsilon}}, \bar{\rho}^{b, (n)}_0{)}.
		\end{align}
		This implies, {with $A_{\varepsilon} := \bigcup_{n \in \N} [{\rho^{b+\varepsilon, (n)}_{b+\varepsilon}}, \bar{\rho}^{b, (n)}_0)$,
		\begin{align}
		\begin{split}
		&v^w_{b+\varepsilon}(x)-v^w_b(x) \\
	&= 	\E_x \left[ \int_{A_{\varepsilon}}e^{-\alpha t} (w( {U^{(0,b)}(t)}+\varepsilon)-w( {U^{(0,b)}(t)})) \diff t \right] + \E_x \left[ \int_{A_{\varepsilon}^c}e^{-\alpha t} (w( {U^{(0,b+\varepsilon)}(t)})-w( {U^{(0,b)}(t)})) \diff t \right] \\
		 &\geq	\E_x \left[\int_{A_{\varepsilon}}e^{-\alpha t} (w( {U^{(0,b)}(t)}+\varepsilon)-w( {U^{(0,b)}(t)})) \diff t \right],
			 \quad x\in\R,
			 \end{split}
			 \label{10}
		\end{align}
		where the last inequality holds by \eqref{6} and the assumption that $w$ is non-decreasing.}
		
		{Similarly, }
		by \eqref{6} and since $ {U^{(0,b+\varepsilon)}(\rho^{b+\varepsilon, (n)}_{0})}=0$ {(see \eqref{rho_b_jump})}  {and $U^{(0,b)} \geq 0$ uniformly in time}, 
		\begin{align}
			 {U^{(0,b+\varepsilon)}(\rho^{b+\varepsilon, (n)}_{0})-U^{(0,b)}(\rho^{b+\varepsilon, (n)}_{0})}=0, \label{8}
		\end{align}
		{which implies}
		\begin{align}
			 {U^{(0,b+\varepsilon)}(t) -U^{(0,b)}(t)} =0, \quad t\in [{\rho^{b+\varepsilon, (n-1)}_{0}}, \bar{\rho}^{b, (n)}_b{)}. \label{15}
		\end{align}
		By \eqref{6} and \eqref{15}, we have, {with $B_{\varepsilon} := \bigcup_{n \in \N} [\bar{\rho}^{b ,(n)}_{b},{\rho}_0^{b+\varepsilon,(n)})$,} 
		\begin{align}
		\begin{split}
			v^w_{b+\varepsilon}(x)-v^w_b(x) 
			& = 	
			 \E_x \left[ \int_{B_{\varepsilon}}e^{-\alpha t} (w( {U^{(0,b+\varepsilon)}(t)})-w( {U^{(0,b)}(t)})) \diff t \right] \\
			 &\leq \E_x \left[
			{\int_{B_{\varepsilon}}}
			e^{-\alpha t}(w( {U^{(0,b)}(t)}+\varepsilon)-w( {U^{(0,b)}(t)})) \diff t \right]
			, \quad x\in\R.
			\end{split}
			\label{11}
		\end{align}
		
		Note that $\rho^{b+\varepsilon, (n)}_{b+\varepsilon} \to \rho_b^{b,(n)}$, $\bar{\rho}^{b, (n)}_0 \to \rho_0^{b,(n)}$, $\bar{\rho}^{b ,(n)}_{b} \to   \rho_b^{b,(n)}$, ${\rho}_0^{b+\varepsilon,(n)} \to  \rho_0^{b,(n)}$ {as $\varepsilon \downarrow 0$} {a.s.} 
		by a similar argument to that of the proof of \cite[Lemma 1]{NobYam2020} and by induction.

		{{Because $w$ is concave,} 
		\[
		\frac{w( {U^{(0,b)}(t)}+\varepsilon)-w( {U^{(0,b)}(t)})}{\varepsilon}\leq w_+^\prime(0), \qquad  t\geq0.
		\]
		Hence, using \eqref{10} and \eqref{11}, the fact that $\bE_x\left[\int_0^\infty e^{-\alpha t}w_+^\prime(0) \diff t\right]<\infty$, and dominated convergence gives, 
		by taking the limit as $\varepsilon \downarrow0$,} 
		\begin{align}
			\lim_{\varepsilon\downarrow0}\frac{v^w_{b+\varepsilon}(x)-v^w_b(x)}{\varepsilon}
			=&\E_x \left[\sum_{n\in\N}\int_{{\rho}^{b ,(n)}_{b}}^{{\rho}^{b,(n)}_{0}}e^{-\alpha t} w_+^\prime ( {U^{(0,b)}(t)})\diff t \right]\\
			=&\sum_{n\in\N}\E_x\left[ e^{-\alpha \rho^{b,(n)}_b}\right]\E_b \left[\int_0^{\kappa^{b,-}_0}e^{-\alpha t} w_+^\prime ( {U^{(0,b)}(t)})\diff t  \right]\\
			=&\sum_{n\in\N}\E_x\left[ e^{-\alpha \rho^{b,(1)}_b} \right]{\left( \E_b\left[e^{-\alpha\kappa^{b,-}_0}\right]\E_0 \left[e^{-\alpha\rho^{b, (1)}_b}\right]\right)}^{n-1}\E_b \left[\int_0^{\kappa^{b,-}_0}e^{-\alpha t} w_+^\prime ( {U^{(0,b)}(t)})\diff t \right]\\
			=&\frac{\E_x \left[e^{-\alpha\rho^{b, (1)}_b}\right]\E_b \left[\int_0^{\kappa^{b,-}_0}e^{-\alpha t} w_+^\prime ( {U^{(0,b)}(t)})\diff t \right]}{1-\E_b\left[e^{-\alpha\kappa^{b,-}_0}\right]\E_0 \left[e^{-\alpha\rho^{b, (1)}_b}\right]}, \qquad x\in\R. \label{14}
		\end{align}
		{Note that the second equality holds by \eqref{rho_b_jump} and the strong Markov property.}
		By {summing} \eqref{12}, \eqref{13} and \eqref{14}, {and the fact that $ {Y^b(t)}= {U^{(0,b)}(t)}$ for $t\in[0,\kappa_0^{b,-}) $} 
		we have 
		\begin{align}
			\lim_{\varepsilon \downarrow 0}\frac{v_{b+\varepsilon}(x)-v_b(x)}{\varepsilon}
			=\frac{\E_x \left[e^{-\alpha\rho^{b, (1)}_b}\right]}{1-\E_b\left[e^{-\alpha\kappa^{b,-}_0}\right]\E_0 \left[e^{-\alpha\rho^{b, (1)}_b}\right]}
			(g(b)-1), \quad x\in\R, b>0.\label{18}
		\end{align}
		This also shows that the mapping $b \mapsto v_b (x)$ is right-continuous on $(0, \infty)$. By the same argument as that of the proofs of \eqref{12}, \eqref{13} and \eqref{14}, it easy to check that $b \mapsto  v_b(x)$ is left continuous on $(0, \infty)$.
		
		{Since the mapping $b \mapsto g(b)$ is non-increasing and using the definition of $b^*$ as in \eqref{threshold}, we obtain that}
		\begin{align}\label{18_a}
			g(x)
			\begin{cases}
				\geq 1, \qquad &x< b^\ast,\\ 
				\leq 1, \qquad &x\geq b^\ast.
			\end{cases}
		\end{align}
		{Therefore by the continuity of the mapping $b \mapsto v_b(x)$, identity \eqref{18_a}, 
		and the same argument} as that of the proof of \cite[Theorem 3.1]{NobYam2020}, the inequality {\eqref{v_b_max} holds over all $b > 0$.}

		It is now left to show that the result extends to $b=0$ for the case of bounded variation. As in the proof of \cite[Lemma 4.3]{Nob2019}, we have $v_b^R(x) \rightarrow v_0^R(x)$ and $v_b^L(x) \rightarrow v_0^L(x)$ as $b \downarrow 0$ for each $x \geq 0$. In addition, because $U^{(0,b)}(t) \xrightarrow{b \downarrow 0} U^{(0,0)}(t)$ a.s. for each $t \geq 0$ (again by  \cite[Lemma 4.3]{Nob2019}) dominated convergence gives $v_b^w(x) \rightarrow v_0^w(x)$ as well. Hence $v_{b^*}(x) \geq v_b(x) \xrightarrow{b \downarrow 0}  v_0(x)$.
			\end{proof}
		
\section{Verification of optimality} \label{section_verification} 

In this section, we upgrade the optimality given in Lemma \ref{lemma_opt_barrier} by showing that it is indeed optimal over all admissible strategies $\Pi$. The proof is via Ito's formula, which requires smoothness of the candidate value function. To this end, we assume the following in addition to Assumptions  \ref{AA3} and \ref{AA1}. 
It is noted that this is only assumed for the unbounded variation case and  is not restrictive in view of  Lemma  \ref{lemma_example_condition}, whose proof is deferred to Appendix \ref{Sec00B}.


	\begin{assump}\label{Ass306}
		When $X$ has unbounded variation paths, we assume, for {any} 
		locally bounded measurable function $f$, $\theta> 0$ and $b > 0$, that the function $H^{(b,\theta)}_f$, {given by
			\begin{align}
				H^{(b, \theta)}_f (x):= \E_x \left[\int_0^{\kappa^{b, -}_0}e^{-\theta t} f({Y^b(t)}) \diff t  \right], \quad {x \in \R,}
		\end{align}}
		has an {a.e.-continuous Radon--Nikodym density locally bounded on $(0, b]$, } where $Y^b$ denotes the reflected \lev process defined in \eqref{A001}.
	\end{assump}

\begin{lemma} \label{lemma_example_condition} If the \lev measure satisfies $\nu (0, \infty) <\infty $ or $\nu (-\infty,0) <\infty $, Assumption \ref{Ass306} is satisfied. 
\end{lemma}

	\begin{remark}\label{Rem401}
		From Assumption \ref{Ass306} and by {the fact that}
			$H^{(b, \theta)}_1(x)= \frac{1}{\theta}\left( 1- \E_x \left[\exp(-\theta \kappa^{b, -}_0 ) \right]\right)$,
			$\theta>0$, 
		the function $x \mapsto \E_x \left[ \exp(-\theta \kappa^{b, -}_0) \right]$ 
		has {a.e.-continuous Radon--Nikodym density locally bounded on $(0, b]$, } when $X$ has unbounded variation paths. 
	\end{remark}
	


We now state the main result of this paper.
\begin{theorem}\label{Thm401}
	Under Assumptions \ref{AA3},  \ref{AA1}, and \ref{Ass306},
the {double} barrier strategy $\pi^{(0,b^*)}$, {with $b^*$ given by \eqref{threshold}}, is optimal and the value function is given by $v(x)=v_{b^*}(x)=v_{\pi^*}(x)$ for all $x\geq0$.
\end{theorem}}
The remainder of this section is devoted to proving Theorem \ref{Thm401}.

	\subsection{Verification lemma}\label{Subsec_Aux_Verification}
	We first state a verification lemma, which provides sufficient conditions for a strategy to be optimal.

	First, let $C^{(1)}_{\text{line}}$ be the set of continuous {functions} $f:\bR\to\bR$ 
	{of at most linear growth and admits a continuous derivative on $(0,\infty)$, i.e.,}
	\begin{enumerate} 
	\item
	There exist $a_1, a_2 >0$ such that  
	$
	|f(x)|<a_1 +a_2 |x|
	$
	for all $x\in\bR$;
	\item {$f \in C^1(0, \infty)$.}  
	\end{enumerate}
	Second, let $C^{(2)}_{\text{line}}$ be a subset of $C^{(1)}_{\text{line}}$ 
	such that the {(continuous)} derivative $f^\prime$ admits, 
	on the positive half line, {an a.e.-continuous and uniformly bounded 
	density function $f^{\prime\prime}: (0, \infty)\to [0, \infty)$ 
	satisfying $f^\prime (0)+\int_0^x f^{\prime\prime}(y) \diff y = f^\prime(x)$ for $x>0$.}

	For $f \in C^{(1)}_{\text{line}}$ {(resp.\ $C^{(2)}_{\text{line}}$) 
	(we fix the density $f^{\prime\prime}$ of $f^\prime$ on $(0, \infty)$ when $X$ has unbounded variation paths)}, 
	define the operator
	\begin{align}
		\mathcal{L} f(x) := \gamma f^\prime(x) +\frac{1}{2}\sigma^2 f^{\prime\prime}(x) 
		+\int_{\R \backslash \{0\} } (f(x+z)-f(x)-f^\prime(x) z\Ind_{\{|z|<1\}}) \nu(\diff z) , \quad x\in (0,\infty).
	\end{align}
	
	
	{A verification lemma is given as follows. We omit the proof because it is similar to that of \cite[Proposition 5.2]{Nob2019}.}
	
	\begin{lemma}\label{verification}
		Let {$v: \R \to \R$} belonging to $C^{(1)}_{\text{line}}$ when $X$ is of bounded variation and $C^{(2)}_{\text{line}}$ otherwise and satisfies
		\begin{align}
			(\mathcal{L} -\alpha )v(x)+ rw(x)&\leq 0, \quad x\in(0, \infty),\label{45} \\ \label{4}
			0\leq v^\prime (x) &\leq \beta, \quad x \in (0, \infty). 
		\end{align}
		Then we have $v(x) \geq \sup_{\pi\in{\Pi}}v_\pi(x)$ for all $x \in \R$. 
	\end{lemma}



	\subsection{Derivative of $v_{b^*}$ 
	}
	
	
	For the next result, we define {the} hitting times to half-lines of the \lev process $X$ as
	\begin{align}
	\tau^-_x =\inf \{t> 0 :{X(t)} <x \},\quad \tau^+_x =\inf \{t> 0 :{X(t)} >x \}, \quad x\in\bR. 
	\end{align}

	\begin{lemma}\label{Lem304_v1}
	Fix $b>0$ (resp., $b\geq 0$) 
		when $X$ has {unbounded} (resp., {bounded}) variation paths. 
		The function $x \mapsto v_b(x)$ is continuous on $\bR$ and {continuously} differentiable on $\bR\backslash\{0, b\}$, {with} 
		\begin{align}\label{density_v_1}
			v_{ b}^{\prime}(x) &= \E_x \left[e^{-\alpha \tau^+_b}; \tau^+_b<{\tau^{-}_0} \right] +\beta \E_x\left[e^{-\alpha{\tau^{-}_0}};{\tau^{-}_0}<\tau^+_b\right]
			+r\E_x \left[\int_0^{\tau^{+}_b \land {\tau^{-}_0}} e^{-\alpha t} w^\prime_+ ( {X(t)}) \diff t\right], 
		\end{align}
		{which also holds when the last expectation is replaced with  $\E_x \left[\int_0^{\tau^{+}_b \land \tau^-_0} e^{-\alpha t} w^\prime_- ( {X(t)}) \diff t\right]$ where $w'_-$ is the left derivative of $w$.}
	\end{lemma}

	\begin{proof}
		On $(-\infty , 0) \cup (b, \infty)$, $v_b$ is {differentiable} and the derivative is 
			\begin{align}
				v^\prime_b (x)=
				\begin{cases}
					\beta, \quad &x<0, \\
					1 \quad &x>b,
				\end{cases}
				\label{50}
			\end{align}
{which agrees with \eqref{density_v_1}}. 
	
	We now focus on $(0,b)$.
		{By Remark \ref{Rem401} and the proof of \cite[Lemma 5.3]{Nob2019} with  \cite[Lemma 1]{NobYam2020} (thanks to Assumption \ref{assumption_compound_poisson}), 
		$v_b^L$ and $v_b^R$} 
			are {differentiable} on $(0, b)$ with their derivatives, 
			respectively, by
		\begin{align}
			v_b^{L,\prime} (x):= \E_x \left[e^{-\alpha \tau^+_b}; \tau^+_b<\tau^-_0\right], \quad
			v_b^{R,\prime} (x):= -\E_x \left[e^{-\alpha \tau^-_0}; \tau^-_0<\tau^+_b\right].
			\label{34_a}
		\end{align}
		In addition, $v_b^L$ and $v_b^R$ are continuous on $\bR$ by the proof of \cite[Lemma 5.4]{Nob2019}.
		
		
		\par
		We now compute the derivative of $v_b^w$ on $(0, b)$.
		For $\varepsilon\in\R$ and $t \geq 0$, we write  {$X^{(\varepsilon)}(t)=X(t)+\varepsilon$}. 
		We write $U^{(\varepsilon)}$ for the {doubly} reflected process {with boundaries $0$ and $b$ driven by  $X^{(\varepsilon)}$.}
		 {Additionally, w}e write $\tau^{(\varepsilon),-}_0:=\inf \{t> 0 :  {X^{(\varepsilon)}(t)}<0\}$ and $\tau^{(\varepsilon),+}_b:=\inf \{t> 0 :  {X^{(\varepsilon)}(t)}>b\}$.  
		 {We note,} {for $x \in \bR$,}
		\begin{align}
			\frac{v_b^w(x+\varepsilon)-v_b^w(x)}{\varepsilon}=
			\E_x \left[\int_0^\infty e^{-\alpha t} \frac{w( {U^{(\varepsilon)}(t)})-w( {U^{(0)}(t)})}{\varepsilon}\diff t \right].\label{24}
		\end{align}
		From \cite[(5.5)--(5.7)]{Nob2019}, we have 
		\begin{align}
			 {U^{(\varepsilon)}(t) -U^{(0)}(t)} \in [0, \varepsilon], \quad t\geq 0,\label{21}
		\end{align}
           and it is non-increasing.
		In particular, \cite[(5.7), (5.8) and (5.11)]{Nob2019} imply
		\begin{align}
			 {U^{(\varepsilon)}(t) -U^{(0)}(t)}=\varepsilon, \quad t\in[0, \tau^{(\varepsilon),+}_b \land \tau^{(0),-}_0],
		\end{align}
		and, so using the fact that $w$ is non-decreasing, we have 
		\begin{align}
		{v_b^w(x+\varepsilon)-v_b^w(x)} \geq \E_x \left[\int_0^{\tau^{(\varepsilon),+}_b \land \tau^{(0),-}_0} e^{-\alpha t} \left({w( {U^{(0,b)}(t)}+\varepsilon)-w( {U^{(0,b)}(t)})}\right)\diff t \right].
			\label{22}
		\end{align}
		On the other hand, from \cite[Section C]{Nob2019}, 
 by the equation after (5.8) and the equation following (5.11), we have $U^{(\varepsilon)}(\tau^{(0),+}_b \land \tau^{(\varepsilon),-}_0) -U^{(0)}(\tau^{(0),+}_b \land \tau^{(\varepsilon),-}_0)=0$. Since $t\mapsto U^{(\varepsilon)}(t) -U^{(0)}(t)$ is non-increasing, we have 
		\begin{align}
			 {U^{(\varepsilon)}(t) -U^{(0)}(t)}=0, \quad t\geq\tau^{(0),+}_b \land \tau^{(\varepsilon),-}_0. \label{20}
		\end{align}
		From \eqref{21} and \eqref{20}, we have 
		\begin{align}
			{v_b^w(x+\varepsilon)-v_b^w(x)}
			\leq \E_x \left[\int_0^{\tau^{(0),+}_b \land \tau^{(\varepsilon),-}_0} e^{-\alpha t} \left({w( {U^{(0,b)}(t)}+\varepsilon)-w( {U^{(0,b)}(t)})}\right)\diff t \right].
			\label{23}
		\end{align}
		
		From \eqref{24}, \eqref{22} and \eqref{23}, 
		{and by taking limit as $\varepsilon\downarrow0$, dominated convergence gives} for $x \in (0, b)$
		\begin{align*}
		{(v^{w}_b)_+' (x):=}
			\lim_{\varepsilon\downarrow0}\frac{v_b^w(x+\varepsilon)-v_b^w(x)}{\varepsilon}=
			\E_x \left[\int_0^{\tau^{+}_b \land \tau^-_0} e^{-\alpha t} w^\prime_+ ( {U^{(0,b)}(t)})\diff t \right]
			=\E_x \left[\int_0^{\tau^{+}_b \land \tau^-_0} e^{-\alpha t} w^\prime_+ ( {X(t)}) \diff t\right],
		\end{align*}
		{because} $\lim_{\varepsilon \downarrow 0}\tau^{(\varepsilon),-}_0=\tau^-_0$ and $\lim_{\varepsilon \downarrow 0}\tau^{(\varepsilon),+}_b=\tau^+_b$ $\p_x$-a.s. for $x\in(0, b)$ by 
		\cite[Lemma 1(ii)]{NobYam2020}.
	{

{
		By changing from $0$ to $-\varepsilon$ and from $\varepsilon$ to $0$ in \eqref{24}, \eqref{22} and \eqref{23}, we have 
		\begin{multline}\label{49}
		\E_x \left[\int_0^{\tau^{(0),+}_b \land \tau^{(-\varepsilon),-}_0} e^{-\alpha t} \left({w( {U^{(0,b)}(t)})-w( {U^{(0,b)}(t)-\varepsilon})}\right)\diff t \right]\\
		\leq v_b^w (x)-v_b^w (x - \varepsilon)
		=
		\E_x \left[\int_0^\infty e^{-\alpha t} \left({w( {U^{(0)}(t)})-w( {{U^{(-\varepsilon)}(t)}})}\right)\diff t \right] \\
		\leq
		\E_x \left[\int_0^{\tau^{(-\varepsilon),+}_b \land \tau^{(0),-}_0} e^{-\alpha t} \left({w( {U^{(0,b)}(t)})-w( {U^{(0,b)}(t)-\varepsilon})}\right)\diff t \right], 
		\end{multline}
		and thus we obtain, for $x \in (0, b)$, 
		\begin{multline}
		(v^{w}_b)_-' (x):=
			\lim_{\varepsilon\downarrow0}\frac{v_b^w(x)-v_b^w(x-\varepsilon)}{\varepsilon}
			=\lim_{\varepsilon \downarrow 0}\E_x \left[\int_0^\infty e^{-\alpha t} \frac{w( {U^{(0)}(t)})-w( {U^{(-\varepsilon)}(t)})}{\varepsilon}\diff t \right]\\
			=\E_x \left[\int_0^{\tau^{+}_b \land \tau^-_0} e^{-\alpha t} w^\prime_- ( {U^{(0,b)}(t)})\diff t \right]
			=\E_x \left[\int_0^{\tau^{+}_b \land \tau^-_0} e^{-\alpha t} w^\prime_- ( {X(t)}) \diff t\right]. 
		\end{multline}
		}	
	}
	
Note that the inequalities \eqref{22}, \eqref{23} and \eqref{49} hold for $x\in\bR$ (also at $0$ and $b$), and hence
$v_b^w$ is continuous on $\bR$ by the dominated convergence theorem.
	
	{Since the potential measure of $X$ does not have a mass by Assumption \ref{assumption_compound_poisson} and \cite[Proposition I.15]{Ber1996}, we have, for $x\in (0, b)$, 
	{{that} $v^w_b$ is differentiable with}
	\begin{align}
	{v^{w,\prime}_b(x)} 
	=\E_x \left[\int_0^{\tau^{+}_b \land \tau^-_0} e^{-\alpha t} w^\prime_+ ( {X(t)}) \diff t\right] = \E_x \left[\int_0^{\tau^{+}_b \land \tau^-_0} e^{-\alpha t} w^\prime_- ( {X(t)}) \diff t\right].\label{55}
	\end{align}

	}

From \eqref{34_a}, \eqref{55}, the dominated convergence theorem and  \cite[Lemma 1(ii)]{NobYam2020}, we have, for $x\in (0, b)$, 
\begin{align}
&v_b^{L,\prime} (x+)= v_b^{L,\prime} (x-)=\E \left[e^{-\alpha \tau^+_{b-x}}; \tau^+_{b-x}<\tau^-_{-x}\right]=v_b^{L,\prime} (x), \\
	&		v_b^{R,\prime} (x+)=v_b^{R,\prime} (x-)= -\E \left[e^{-\alpha \tau^-_{-x}}; \tau^-_{-x}<\tau^+_{b-x}\right]=v_b^{R,\prime} (x), \\
	&v^{w,\prime}_b (x+)=\E \left[\int_0^{\tau^{+}_{b-x} \land \tau^-_{-x}} e^{-\alpha t} w^\prime_+ ( X(t)+x) \diff t\right], \\
		&v^{w,\prime}_b (x-)=\E \left[\int_0^{\tau^{+}_{b-x} \land \tau^-_{-x}} e^{-\alpha t} w^\prime_- ( X(t)+x) \diff t\right], 
\end{align}
{where the last two equations hold by \eqref{55}.}
Therefore, the continuity of $v_b^\prime$ on $(0, b)$ is obtained.

	
		\end{proof}

{In the following lemma, we show that the function $v_{b^*}$ with our selection of $b^*$ is smooth also at $b^*$.}
Because {we consider a general \lev process $X$} and the function $w$ is allowed to be  
	non-differentiable, we use the following technique to obtain an useful expression of the derivative $v_{b^*}'$.
This is only {a} technicality and {for instance, it is not required for the case $X$ is of unbounded variation and $w$ is continuously differentiable,} and Lemma \ref{Lem304_a} below
	holds with $w'$.
	
		{Fix} $p\in[0, 1]$. We define a Bernoulli random variable $A_p$, with
	\begin{align}
		A_p=
		\begin{cases}
			0, \quad \text{with probability } p,\\
			1, \quad \text{with probability } 1-p,
		\end{cases}
	\end{align}
 independent of $X$. 
	{For the single-sided reflected process $Y^{b^*}$}, we define $\kappa^{[p]}_0$, parameterized by $p \in [0,1]$ as a modification of $\kappa^{b^\ast, -}_0$ (see  \eqref{kappa_def}) as follows {(recall the definition of regularity in Lemma \ref{Lem402})}:
	\begin{enumerate}
		\item 
		When $0$ is regular for $(-\infty, 0)$, we set $\kappa^{[p]}_0:= \kappa^{b^\ast, -}_0$;
		\item 
		When $0$ is irregular for $(-\infty , 0)$, {which occurs in the case that} $X$ has bounded variation paths and a non-negative drift, 
		we set
		\begin{align}
			\kappa^{[p]}_0:= \kappa^{b^\ast, -}_0\Ind_{\{A_p=1\}}+K^{b^\ast, -}_0\Ind_{\{A_p=0\}},
		\end{align}
		where
		$K^{b^\ast, -}_0 := \inf \{t\geq 0 :  {Y^{b^\ast}(t)} \leq 0 \}$.
	\end{enumerate}
	We also define 
	\begin{align}
		w^\prime_p(x):=(1-p)w^\prime_+(x)+pw^\prime_-(x), \quad x >  0.
	\end{align}
	Let
			\begin{align*}
		g_p(b^*) := \beta\mathbb{E}_{b^\ast}\left[e^{-{{\alpha }}\kappa^{[p]}_0}\right]
			+r\mathbb{E}_{b^\ast}\left[\int_0^{\kappa^{[p]}_0}e^{-\alpha t}w'_{p}( {Y^{b^\ast}(t)}) \diff t\right] = \beta - \E_{b^*} \left[ \int_0^{\kappa^{[p]}_0} e^{-\alpha t} l_p ( {Y^{b^*}(t)}) \diff t  \right],
		\end{align*}
		where $l_p(y) := \beta \alpha -rw'_p (y)$ for $y>0$.

	Then we have the following {result} {whose proof is deferred to Appendix \ref{AppA03}}.  
	\begin{lemma}\label{aux_1}
	Suppose that $b^\ast>0$.
		There exists $p^\ast \in [0, 1]$ which satisfies 
		\begin{align}
			g_{p^*}(b^*) = 1. \label{26}
		\end{align}
	\end{lemma}

	\begin{lemma}\label{Lem304_a}
		{The function $v_{b^*}$ is {continuously} {differentiable} on $\bR\backslash\{0\}$. In addition,} 
{for all ${x\in\bR\backslash\{0\}}$,}
		\begin{align}
			v_{b^\ast}^\prime (x) &=\beta \E_x\left[ e^{-\alpha \kappa^{[p^\ast]}_0} \right]+r\E_x \left[ \int_0^{\kappa^{[p^\ast]}_0}e^{-\alpha t}w^\prime_{p^\ast}( {Y^{b^\ast}(t)})\diff t \right] \\
			&=\beta - \E_x \left[ \int_0^{\kappa^{[p^\ast]}_0}e^{-\alpha t}{l_{p^\ast}}( {Y^{b^\ast}(t)}) \diff t \right],  
			\label{47}
		\end{align}
		{where $p^\ast$ is given in Lemma \ref{aux_1}}.
%
	\end{lemma}
	\begin{proof}
	
	We define 
		\begin{align}
			\tau^{[p^*]}_0:= \tau^{-}_0\Ind_{\{A_{p^*}=1\}}+T^{-}_0\Ind_{\{A_{p^*}=0\}} \quad \textrm{where } \; T^-_0 := \inf\{t >0 : X(t) \leq 0\}.
		\end{align} 
	Additionally, consider \begin{align}
		\label{39a}
		A(x)
		&:= 
		{a_1}(x)
		+\beta{a_2}(x)
		+r {a_3}(x), \qquad x\in {\bR}, 
\end{align}
where 
\begin{align*}
{a_1}(x) &:= \E_x \left[e^{-\alpha \tau^+_{b^*}}; \tau^+_{b^*}<\tau^{{[p^\ast]}}_0 \right], \\
{a_2}(x) &:= \E_x\left[e^{-\alpha \tau^{{[p^\ast]}}_0};\tau^{{[p^\ast]}}_0<\tau^+_{b^\ast}\right], \\
{a_3}(x) &:= \E_x \left[\int_0^{\tau^{+}_{b^\ast} \land\tau^{{[p^\ast]}}_0} e^{-\alpha t} w^\prime_{{p^\ast}} ( {X(t)})\diff t\right].
\end{align*}
For $\varepsilon > 0$ and $x > 0$, we have $\tau_{\varepsilon}^- \leq T^-_0\leq \tau^{[p^*]}_0 \leq \tau^-_0$, $\mathbb{P}_x$-a.s. Thus taking $\varepsilon \downarrow 0$ and thanks to the continuity as in \cite[Lemma 1(i)]{NobYam2020}, we have
		\begin{align} \label{tau_same}
		\tau^-_0=T^-_0=\tau^{[p^*]}_0 \quad \mathbb{P}_x-a.s., \textrm{ for } x > 0.
		\end{align}
		Therefore, as  a corollary of Lemma \ref{Lem304_v1},
	for $x \in \bR\backslash\{0, b^*\}$,
		\begin{align}\label{density_v}
			v_{b^*}^{\prime}(x) =A(x). 
		\end{align}
								By the definition of the single reflected L\'evy processes, we have 
		$\{Y^{b^\ast}_t : t< \tau^+_{b^\ast}\} =\{X_t : t< \tau^+_{b^\ast}\}$, and so we have 
		$\tau^{[p^\ast]}_0= \kappa^{[p^\ast]}_0$ on $\{\tau^{[p^\ast]}_0< \tau^+_{b^\ast}\}$, which is equal to $\{\kappa^{[p^\ast]}_0< \tau^+_{b^\ast}\}$. 
		
By these 
and the strong Markov property (noting $Y^{b^*}(\tau_{b^*}^+) = b^*$ {on $\{\tau_{b^*}^+ < \infty\}$}), we have 
				\begin{align*}
		{a_2}(x) 
		= \E_x\left[e^{-\alpha \kappa^{[p^\ast]}_0};\kappa^{[p^\ast]}_0<\tau^+_{b^\ast}\right]  
		&= \E_x\left[e^{-\alpha \kappa^{[p^\ast]}_0} \right] - \E_x\left[e^{-\alpha \kappa^{[p^\ast]}_0};\kappa^{[p^\ast]}_0> \tau^+_{b^\ast}\right]\notag\\
		&= \E_x\left[e^{-\alpha \kappa^{[p^\ast]}_0} \right] -
		{a_1}(x) 
		\E_{b^*}\left[e^{-\alpha \kappa^{[p^\ast]}_0} \right],
		\end{align*}
and 
						\begin{align*}
{a_3}(x)
&= \E_x \left[\int_0^{\kappa^{[p^\ast]}_0} e^{-\alpha t} w^\prime_{p^\ast} ( {Y^{b^*}(t)})\diff t\right] -
		\E_{x} \left[ \Ind_{\{\tau_{b^*}^+ < {\kappa^{[p^\ast]}_0}\}}\int_{\tau_{b^*}^+}^{\kappa^{[p^\ast]}_0 } e^{-\alpha t} w^\prime_{p^*}( {Y^{b^\ast}(t)})\diff t\right] \\
		&= \E_x \left[\int_0^{\kappa^{[p^\ast]}_0} e^{-\alpha t} w^\prime_{p^\ast} ( {Y^{b^*}(t)})\diff t\right] -
		{a_1}(x) 
		\E_{b^\ast} \left[ \int_0^{\kappa^{[p^\ast]}_0 } e^{-\alpha t} w^\prime_{p^*}( {Y^{b^\ast}(t)})\diff t\right].
		\end{align*}
		
Substituting these in \eqref{39a}, {for $x\in\bR$, }
				\begin{align*}
		{A(x)}&={a_1}(x) +\beta 
		\left( \E_x\left[e^{-\alpha \kappa^{[p^\ast]}_0} \right] - {a_1}(x) \E_{b^*}\left[e^{-\alpha \kappa^{[p^\ast]}_0} \right]  \right) \\
			&+r
			\left( \E_x \left[\int_0^{\kappa^{[p^\ast]}_0} e^{-\alpha t} w^\prime_{p^\ast} ( {Y^{b^*}(t)})\diff t\right] - {a_1}(x) \E_{b^\ast} \left[ \int_0^{\kappa^{[p^\ast]}_0 } e^{-\alpha t} w^\prime_{p^*}( {Y^{b^\ast}(t)})\diff t\right] \right) \\
			&={a_1}(x) + \left( \beta -  \E_x \left[\int_0^{\kappa^{[p^\ast]}_0} e^{-\alpha t} {l_{p^\ast}} ( {Y^{b^*}(t)})\diff t\right]\right) - {a_1}(x) \left( \beta -  \E_{b^*} \left[\int_0^{\kappa^{[p^\ast]}_0} e^{-\alpha t} {l_{p^\ast}} ( {Y^{b^*}(t)})\diff t\right]\right) \\
			&{=\beta - \E_x \left[ \int_0^{\kappa^{[p^\ast]}_0}e^{-\alpha t}{l_{p^\ast}}( {Y^{b^\ast}(t)}) \diff t \right]}
		\end{align*}
		by the definition of $b^*$ as a solution to \eqref{26}. {On the other hand, 
		{the equation \eqref{density_v}}
		implies that $v'_{b^*}(x)=A(x)$ for $x\in\bR\backslash\{0, b^\ast\}$. Hence, \eqref{47} holds true on $\bR\backslash\{0, b^\ast\}$}.

{
Since $v_{b^\ast}$ is continuous {on $\mathbb{R}$} and {$v'_{b^*}(x)=A(x)$ for $x\in\bR\backslash\{0, b^\ast\}$, the function $A$} is a density of the function $v_{b^\ast}$. 
By {Lemma \ref{Lem304_v1} and \eqref{density_v},}
the function $v_{b^\ast}^\prime(x)={A(x)}$ is continuous on $(0, b^\ast)\cap (b^\ast, \infty)$, and so 
it is enough to prove the continuity of $x\mapsto {A(x)}$ at $b^\ast$ when $b^\ast> 0 $. {Indeed, this readily shows the existence of $v_{b^*}'(b^*)$ and it coincides with $A(b^*)$.}

By \eqref{50} and since  ${A}(b^*)=1$ by Lemma \ref{aux_1}, the function $x\mapsto {A(x)}$ 
 is right continuous at $b^\ast$. 
Thus, {it remains to confirm} the left continuity {at $b^*$ (i.e.\ $A(b^*-) = A(b^*)=1$)}. 
Since $X$ is not a driftless compound Poisson process, $0$ is regular for $\bR\backslash\{0\}$ {--
meaning at least one of the following holds: (1) $0$ is regular for $(0, \infty)$ and/or (2) $0$ is regular for $(-\infty,0)$.
}

{(1)} If $0$ is regular for $(0, \infty)$ for $X$, we have, for $x \in (0, b^\ast)$, 
\begin{align*}
{A(x)}
=&\beta - \E_x \left[ \int_0^{\tau^+_{b^\ast}\land{\tau^{[p^\ast]}_0 }}e^{-\alpha t}{l_{p^\ast}}( {X(t)}) \diff t \right]
-\bE_x\left[e^{-\alpha\tau^+_{b^\ast}}; \tau^+_{b^\ast}<\tau^{[p^\ast]}_0\right]\E_{b^\ast} \left[ \int_0^{\kappa^{[p^\ast]}_0}e^{-\alpha t}{l_{p^\ast}}( {Y^{b^\ast}(t)}) \diff t \right]\\
=&\beta - \E_{b^\ast} \left[ \int_0^{\tau^+_{2b^\ast-x}\land{\tau^{[p^\ast]}_{b^\ast-x} }}e^{-\alpha t}{l_{p^\ast}}( {X(t)}-b^\ast+x) \diff t \right]\\
&\qquad \qquad 
-\bE_x\left[e^{-\alpha\tau^+_{b^\ast}}; \tau^+_{b^\ast}<\tau^{[p^\ast]}_0\right]\E_{b^\ast} \left[ \int_0^{\kappa^{[p^\ast]}_0}e^{-\alpha t}{l_{p^\ast}}( {Y^{b^\ast}(t)}) \diff t \right]\\
{\xrightarrow{x \uparrow b^\ast}}
&\beta-\E_{b^\ast} \left[ \int_0^{\kappa^{[p^\ast]}_0}e^{-\alpha t}{l_{p^\ast}}( {Y^{b^\ast}(t)}) \diff t \right]=
{1 = A(b^*)}
\end{align*}
where we used 
\cite[Lemma 1(ii), (iii)]{NobYam2020} with $\bE_{b^\ast}\left[e^{-\alpha \tau^+_{b^\ast}}\right]=1$ {(because $0$ is regular for $(0,\infty)$ and by the spatial homogeneity of $X$)} in the limit.

{(2)} If $0$ is regular for $(-\infty , 0)$ for $X$, we have, for $x\in (0, b^\ast)$,
\begin{align}
{1=}{A({b^*})}
=&\beta - \E_{b^*} \left[ \int_0^{\kappa^{[p^\ast]}_0}e^{-\alpha t}{l_{p^\ast}}( {Y^{b^\ast}(t)}) \diff t \right]
\\
=&\beta - \E_{b^\ast} \left[ \int_0^{\kappa^{b^\ast,-}_x}e^{-\alpha t}{l_{p^\ast}}( {Y^{b^\ast}(t)}) \diff t \right]
-\bE_{b^\ast} \left[  e^{-\alpha \kappa^{b^\ast,-}_x}  \bE_{Y^{b^\ast} (\kappa^{b^\ast,-}_x)} \left[ \int_0^{\kappa^{[p^\ast]}_0}e^{-\alpha t}{l_{p^\ast}}( {Y^{b^\ast}(t)}) \diff t\right]\right]\\
\geq &\beta - \E_{b^\ast} \left[ \int_0^{\kappa^{b^\ast,-}_x}e^{-\alpha t}{l_{p^\ast}}( {Y^{b^\ast}(t)}) \diff t \right]
-\bE_{b^\ast} \left[  e^{-\alpha \kappa^{b^\ast,-}_x} \right] \bE_{x} \left[ \int_0^{\kappa^{[p^\ast]}_0}e^{-\alpha t}{l_{p^\ast}}( {Y^{b^\ast}(t)}) \diff t\right]
\\
{\xrightarrow{x \uparrow b^\ast}}&\beta-\lim_{x\uparrow b^\ast}\bE_{x} \left[ \int_0^{\kappa^{[p^\ast]}_0}e^{-\alpha t}{l_{p^\ast}}( {Y^{b^\ast}(t)}) \diff t\right]= \lim_{x\uparrow b^\ast} {A(x)} 
\end{align}
where we used \cite[Lemma 1(ii)]{NobYam2020} with $ e^{-\alpha \kappa^{b^\ast,-}_x} \geq  e^{-\alpha \tau^{-}_{x}} $ and $ \bE_{b^\ast}\left[  e^{-\alpha \tau^{-}_{b^\ast}} \right]=1$ {(because $0$ is regular for $(-\infty,0)$ and by the spatial homogeneity of $X$)} in the limit. 
Since 
{$A(x)$}
 is non-increasing (because $Y^{b^*}(t)$ is monotone in the starting value), we have {${A(b^\ast)} \geq \lim_{x\uparrow b^\ast}{A(x)} \geq A(b^*)$, showing the left-continuity}. 
 
Therefore, the function {$x\mapsto A(x)$} is continuous on $(0, \infty)$ in each case and it is the derivative of $v_{b^\ast}$ on $(0, \infty)$. 
}		
		The proof is complete. 
	\end{proof}
	
	\subsection{Proof of optimality}
	From Lemma \ref{Lem304_a}, we obtain the following lemma. 
	\begin{lemma}\label{Lem408} 
	\begin{enumerate}
		\item The function $v_{b^\ast}$ is concave {on $\R$} and belongs to $C^{(1)}_{\text{line}}$, with its {derivative} $v_{b^\ast}^\prime $ satisfying $1\leq v_{b^\ast}^\prime(x)\leq \beta $ for $x \in \R \backslash\{0\}$. 
\item When $X$ has unbounded variation paths, $v_{b^\ast}$ belongs to $C^{(2)}_{\text{line}}$.
		\end{enumerate}
	\end{lemma}
	\begin{proof}
		(1)  		Since $v_{b^\ast}$ has a {continuous derivative} $v_{b^\ast}^\prime$ {on $\bR\backslash\{0\}$} {({by Lemma \ref{Lem304_a}})}, 
		 {$v_{b^*}$ belongs to $C^{1}_{\text{line}}$}. 
		{By \eqref{50}}, it suffices to prove that $v_{b^\ast}^\prime $ is decreasing on ${\bR}$. 
		

	From {Lemma} \ref{Lem304_a}, {the fact that $l_{p^\ast}$ is non-negative and non-decreasing,}
	and because $Y^{b^*}$ {and $\kappa^{[p^\ast]}_0$ are} monotone in the starting value $x$, the mapping
			$x \mapsto v_{b^\ast}^\prime (x)
			=\beta - \E_x \Big[ \int_0^{\kappa^{[p^\ast]}_0}e^{-\alpha t}{l_{p^\ast}}( {Y^{b^\ast}(t)}) \diff t \Big]$ is non-increasing on $x \in{\bR}$. 
%
%
%
		
		(2)
In addition, $v_{b^\ast}$ belongs to $C^{(2)}_{\text{line}}$ when $X$ has unbounded variation paths by Lemma \ref{Lem304_a}, \cite[Lemma 5.6]{Nob2019} and Assumption \ref{Ass306} and since $\kappa^{[p^\ast]}_0= \kappa^{b^\ast, -}_0$ when $0$ is regular for $(-\infty , 0)$. 


	\end{proof}
We take $v_{b^\ast}^{\prime\prime}$, which is the density of $v_{b^\ast}^{\prime}$, on $(-\infty, 0] \cup (b^\ast, \infty)$, as $0$. 

	\begin{lemma}\label{Lem309}
	When $X$ has bounded variation paths, we have
		\begin{align}
			&(\mathcal{L} -\alpha )v_{b^\ast}(x)+r w(x)= 0, \quad x\in{(0, b^\ast{]}}. \label{42}
		\end{align}
		When $X$ has unbounded variation paths, we can take $v_{b^\ast}^{\prime\prime}$ {on $(0, b^\ast]$} 
		such that \eqref{42} is satisfied.
	\end{lemma}
	\begin{proof}
We write 
	\begin{align}
	v^{LR}_{b^\ast}(x) := v^L_{b^\ast}(x)-\beta v^R_{b^\ast}(x), \quad &x\in \bR.
	\end{align}
	{Proceeding as in the proof of} \cite[Lemma 5.7]{Nob2019}, 
	 we have 
	\begin{align}
	(\mathcal{L} -\alpha)v^{LR}_{b^\ast}(x)=0, \quad x\in(0, b^\ast), \label{43}
	\end{align}
	when $X$ has bounded variation paths. In addition, {when $X$ has paths of unbounded variation,} we can take the density $v^{LR, \prime\prime}_{b^\ast}$ of $v^{LR, \prime}_{b^\ast}$ to satisfy \eqref{43} 
 by \cite[Lemma 5.7]{Nob2019}. 	
	
	{On the other hand}, we define the process $M^w=\{M^w(t) :t\geq 0\}$ as 
	\begin{align}
	M^w(t)  := \int_0^{t\land \tau^-_0 \land \tau^+_{b^\ast}}e^{-\alpha u} w(U^{(0, b^\ast)}(u))\diff u +e^{-\alpha (t\land \tau^-_0 \land \tau^+_{b^\ast})} v^w_{b^\ast} (X(t\land \tau^-_0 \land \tau^+_{b^\ast})), \quad t\geq 0. 
	\end{align}
	The process $M^w$ is an  {$(\mathcal{F}(t))_{t\geq0}$}-martingale since for $s ,t \geq 0 $ with $s \leq t$, we have
	\begin{align}
	\E_x \left[ M^w(t) | \cF(s) \right]=&\Ind_{\{ \tau^-_0 \land \tau^+_{b^\ast} \leq s\}}
	\left(  \int_0^{ \tau^-_0 \land \tau^+_{b^\ast}}e^{-\alpha u} w(U^{(0, b^\ast)}(u))\diff u +e^{-\alpha (\tau^-_0 \land \tau^+_{b^\ast})} v^w_{b^\ast} (X( \tau^-_0 \land \tau^+_{b^\ast}))\right)\\
	&+\Ind_{\{ \tau^-_0 \land \tau^+_{b^\ast} > s\}}\left(  \int_0^{s}e^{-\alpha u} w(U^{(0, b^\ast)}(u))\diff u +e^{-\alpha s} v^w_{b^\ast} (X(s))\right)=M^w(s). 
	\end{align}
Indeed, for the second inequality in the previous identity, we note that the strong Markov property gives
\begin{align*}
	&\E_x\left[M^w(t)\Ind_{\{ \tau^-_0 \land \tau^+_{b^\ast} > s\}}\Big|\mathcal{F}(s)\right]=\E_x\Bigg[\Ind_{\{ \tau^-_0 \land \tau^+_{b^\ast} > s\}}\Bigg(  \int_0^{t\land \tau^-_0 \land \tau^+_{b^\ast}}e^{-\alpha u} w(U^{(0, b^\ast)}(u))\diff u \notag\\&+\E_{x}\left[\int_{t\land \tau^-_0 \land \tau^+_{b^\ast}}^{\infty}e^{-\alpha u}w(U^{(0, b^\ast)}(u))\diff u\Bigg{|}\mathcal{F}(t\land \tau^-_0 \land \tau^+_{b^\ast})\right]\Bigg)\Bigg{|}\mathcal{F}(s)\Bigg]\notag\\
	&=\Ind_{\{ \tau^-_0 \land \tau^+_{b^\ast} > s\}} \E_{x}\left[\int_{0}^{\infty}e^{-\alpha u}w(U^{(0, b^\ast)}(u))\diff u\Bigg{|}\mathcal{F}(s)\right]\notag\\
	&=\Ind_{\{ \tau^-_0 \land \tau^+_{b^\ast} > s\}}\Bigg(  \int_0^{s}e^{-\alpha u} w(U^{(0, b^\ast)}(u))\diff u +\E_{x}\left[\int_{s}^{\infty}e^{-\alpha u}w(U^{(0, b^\ast)}(u))\diff u\Bigg{|}\mathcal{F}(s)\right]\Bigg)\notag\\
	&=\Ind_{\{ \tau^-_0 \land \tau^+_{b^\ast} > s\}}\left(  \int_0^{s}e^{-\alpha u} w(U^{(0, b^\ast)}(u))\diff u +e^{-\alpha s} v^w_{b^\ast} (X(s))\right).
\end{align*}
	Thus, by {Assumption \ref{Ass306}} 
	and the same argument as {in} the proof of \cite[(12)]{BifKyp2010} and \cite[Lemma 5.7]{Nob2019}, we have
	\begin{align}
	(\mathcal{L} -\alpha )v^w_{b^\ast}(x)+ r w(x)= 0, \quad x\in{(0, b^\ast)},\label{44}
	\end{align}
	when $X$ has bounded variation paths. In addition, 
	we can take the density $v^{w, \prime\prime}_{b^\ast}$ of $v^{w, \prime}_{b^\ast}$ to satisfy \eqref{44} {as in the proof of \cite[Lemma 5.7]{Nob2019}}. 
	From \eqref{43} and \eqref{44}, 
	{we obtain \eqref{42} on $(0, b^\ast)$. 
	
	Since $v_{b^\ast}\in C^1 (0, \infty)$ by Lemma \ref{Lem304_a} 
	and \cite[Remark 2.5]{Nob2019}, the map $x\mapsto (\mathcal{L} -\alpha )v_{b^\ast}(x)+r w(x)-\frac{1}{2}\sigma^2v_{b^\ast}^{\prime\prime}(x)$ is continuous on $(0, \infty)$. Thus, \eqref{42} holds when $X$ has bounded variation paths.   In addition, we can take $v_{b^\ast}^{\prime\prime}(b^\ast)$ to satisfy \eqref{42}.  Here, since $\{b^\ast\}$ is a null set of the Lebesgue measure, we can assign any value at $b^\ast$ of $v_{b^\ast}^{\prime\prime}$. Thus, we can take $v_{b^\ast}^{\prime\prime}(b^\ast)$ to satisfy \eqref{42}.
	The proof is complete.
	}
	\end{proof}
	\begin{lemma}\label{Lem410}
	{When $X$ has bounded variation paths, we have}
		\begin{align}
			&(\mathcal{L} -\alpha )v_{b^\ast}(x)+r w(x) \leq 0, \quad x> b^\ast. \label{51}
		\end{align}
	{When $X$ has unbounded variation paths, we can take $v_{b^\ast}^{\prime\prime}$ on $(b^\ast, \infty)$ 
	such that \eqref{51} is satisfied.}
	\end{lemma}
\begin{proof}
	By following the same procedure as \cite[Lemma 5.8]{Nob2019}, it can be shown that, for all $b > b^*$,
	\[
	v_b(x) - v_{b^*} (x) = \E_x \left[ \int_0^\infty e^{-\alpha t} h(U^{(0,b)}(t)) \Ind_{(b^*, \infty)}(U^{(0,b)}(t)) \diff t \right],\label{53}
	\] 
	where  for $x\geq b^*$
	\begin{align}
		h(x)&:= (\mathcal{L} -\alpha )v_{b^\ast}(x)+rw(x) \notag\\
		&=\gamma+\int_{\bR\backslash \{0\}} (v_{b^\ast}(x+z)-(x+\hat{b})-z\Ind_{\{|z|<1\}})\nu(\diff z)
		-\alpha (x+\hat{b})+ r w(x),
	\end{align}
where $\hat{b} := v_{b^\ast}(b^\ast) - b^\ast$.
	By Lemma \ref{Lem408} and Assumption \ref{assum_w}, $h$ is concave. 
	{Note that $h(b^\ast)=0$ by Lemma \ref{Lem309} when $b^\ast>0$. } 
	
	

	{Suppose (to derive a contradiction), {$b^\ast>0$ and} there exists $a > b^*$ such that $h(a) > 0$. 	
		 Setting $b=a$ and by Lemma \ref{lemma_opt_barrier},
	\begin{align} \label{nonegative_v_a_v_b}
	0 \geq v_a(x) - v_{b^*} (x) = \E_x \left[ \int_0^\infty e^{-\alpha t} h(U^{(0,a)}(t)) \Ind_{{(}b^*, \infty)}(U^{(0,a)}(t)) \diff t \right].
	\end{align}
	Then, by the above properties of $h$ {(concavity, $h(b^*) = 0$ and $h(a) > 0$)}, } necessarily $h(x) > 0$ for all $x \in (b^*,a)$. 
	On the other hand, because $U^{(0,a)} \leq a$ a.s. and Lemma \ref{Lem309},  ${\int_0^\infty e^{-\alpha t}}h(U^{(0,a)}(t)) \Ind_{[b^*, \infty)}(U^{(0,a)}(t)) {\diff t} \geq 0$ {(and strictly positive with a positive probability)}, {which contradicts \eqref{nonegative_v_a_v_b}.} Hence $h(b) \leq 0$ for all $b > b^*$ {when $b^\ast>0$}. Thus, \eqref{51} holds.

	{Suppose that $b^\ast=0$.  Then $X$ has bounded variation paths with non-negative drift {and $\nu(-\infty,0)<\infty$} by Remark \ref{Rem202} and thus we have {using \eqref{50}}, for $x \geq 0$, 
	\begin{align}
	h^\prime_+ (x)=&\nu (-\infty , x) (\beta - 1)-\alpha +rw^{\prime}_+(x)\\
	\leq& \nu(-\infty,0) (\beta - 1)-\alpha +rw^{\prime}_+(0)\leq 0, \label{54}
	\end{align}
	where the last inequality comes from \eqref{52}.} 
	{
		Suppose that $h({0})>0$. Then there exists $a>{0}$ such that $h(b)>0$ for $b \in [0, a ]$. 
	We have ${\int_0^\infty e^{-\alpha t}}h(U^{(0,a)}(t)) \Ind_{[b^*, \infty)}(U^{(0,a)}(t)) {\diff t} > 0$ a.s.
	{This contradicts with  \eqref{nonegative_v_a_v_b} with $b^* = 0$
	(which holds by Lemma \ref{lemma_opt_barrier}).
	}
	Thus, it holds $h({0}) \leq 0$. 
	By \eqref{54}, we obtain $h(b)\leq 0$ for all $b>{0}$. 
	}
\end{proof}

	\begin{proof}[Proof of Theorem \ref{Thm401}]
	From Lemma \ref{Lem408}, the function $v_{b^\ast}$ is sufficiently smooth to apply the operator $\mathcal{L}$. 
	In addition, by Lemma \ref{Lem408}, Lemma \ref{Lem309} and Lemma \ref{Lem410}, 
	the function $v_{b^\ast}$ satisfies \eqref{45} and \eqref{4}. 
	Therefore, by Lemma \ref{verification}, the strategy $\pi^{b^\ast}$ is optimal and the proof is complete. 
	\end{proof}

	\section{An application: Bail-out optimal dividend problem with regime switching}\label{sec3}
	In this section, {as} a direct application of the results from previous sections, we consider the bailout dividend problem driven by a Markov additive process (MAP) with two-sided jumps. We will prove the optimality of double barrier strategies, and will also present a way to obtain the optimal barriers through a combination of the results of previous sections and some recursive operators.
	
	\subsection{Markov additive processes}
	Let us consider a bivariate process  {$(X,H)=\{{(X(t),H(t))} : t \geq 0\}$}
	, where the component $H$ is a continuous-time Markov chain with finite state space $E$ and the generator matrix $Q={(q_{ij})}_{i, j\in E}$. 
	When the Markov chain $H$ is in the state $i \in E$, the process $X$ behaves as a L\'evy process $X^i$. 
	At the $n$-th jump of $H$ and if the state of $H$ changes from $i$ to $j$ (with $i\not=j$),
	the process $X$ jumps 
	{with size}
	$J^{(n)}_{i, j}$, {having a common distribution function $F_{ij}: \R \to [0,1]$.}
	 The components $(X^i)_{i\in E}$, $H$, and 
	 {$(J_{ij}^{(n)})_{i,j\in E, n \in \mathbb{N}}$} are assumed to be independent and are defined on some filtered probability space
	$(\Omega , \mathcal{F} , {\bf F} , {\mathbb{P}} )$ where ${\bf F} := \{ {\cF} {(t)}, t \geq 0\}$
	denotes the right-continuous complete filtration jointly
	generated by the processes $(X,H)$. 
	We will denote by ${\mathbb{P}}_{(x, i)}$ the law of the process conditioned on the event $\{ {X(0)=x , H(0) =i} \}$ and by $\E_{(x,i)}$ its associated expectation operator. 

For $i \in E$ \textrm{and} $x\in\R$, we denote by $\bP^i_x$ the law of the \lev process $X^i$ when it starts at $x$ and by $\bbE^i_x$ its corresponding expectation operator. In particular, we {use the notation} $\bP^i = \bP^i_0$ and {$\bbE^i = \bbE^i_0$}. Additionally, we denote the L\'evy measure of $X^i$  by $\nu(i, \cdot)$, that satisfies $	\int_{\R\backslash \{ 0\}}(1\land x^2)  {\nu} (i,\diff x) < \infty$, {$i \in E$}.
	
	\subsection{Bail-out optimal dividend problem with Markov-switching regimes} \label{Sec302}
{Similar to the single-regime model studied in previous sections, } a strategy  is a pair $\pi := \left( {L^{\pi}(t), R^{\pi}(t)}; t \geq 0 \right)$ 
	of non-decreasing, right-continuous, and adapted processes (with respect to the filtration ${\bf F}$) starting at zero where $L^{\pi}$ is the cumulative amount of dividends and $R^{\pi}$ is that of injected capital. {The corresponding controlled process starts at} ${U^\pi(0-)} := x$ and {follows} ${U^\pi(t)} := {X(t)} - {L^\pi(t)} + {R^\pi(t)}$, $t \geq 0$. 
	
	Let 
$\qq:E \to {(0, \infty)}$
	be the Markov-modulated rate of discounting and let
	\[
	\Lambda(t)=\int_0^t\qq(H(s)) \diff s, \quad t \geq 0, 
	\]
	be the cumulative discount.

	Given $\beta>1$, representing the cost per unit of injected capital, we want to maximize
	\begin{align*}
		V_{\pi} (x,i) := \mathbb{E}_{(x,i)} \left[ \int_{[0,\infty)} e^{-\Lambda(t)}   \diff {L^\pi(t)}  -  \beta\int_{[0, \infty)} e^{-\Lambda(t)} \diff {R^\pi(t)} \right], \quad x \geq 0, \quad {i \in E,}
	\end{align*}
	over all admissible strategies $\mathbf{\Pi}$ 
	such that  
	${U^\pi(t)} \geq 0$ a.s.\ uniformly in $t$ and 
	\begin{align}
		\bE_{(x, i)} \left[ \int_{[0, \infty)} e^{-\Lambda(t)} \diff {R^\pi(t)}\right]< \infty.
	\end{align}
	Hence our aim is to find the value function of the problem, i.e.,
	\begin{equation}\label{vf_rs}
		V(x,i):=\sup_{\pi \in \mathbf{\Pi}}V_{\pi}(x,i), \quad x \geq 0 \quad {i \in E.}
	\end{equation}
	

	For the problem with regime-switching in this section, the following assumptions are mandated.
	\begin{assump}\label{AAA1}
		We assume that $(X^j, \bP^j)$ satisfies Assumption \ref{AA1} and \ref{assumption_compound_poisson} for all $j \in E$.
	\end{assump}
	\begin{assump}\label{AAA2}
		For all $i, j\in E$ with $i\neq j$, we assume that $\E[|{J_{ij}^{(1)}}|]<\infty$.
	\end{assump}
%
%

\begin{remark}	
By Lemma \ref{lemma_example_condition}, Assumption \ref{AAA1} is satisfied if $\nu (j,(0, \infty)) <\infty $ or $\nu (j,(-\infty,0)) <\infty $, {and $\int_{\R\backslash[-1,1]}|x| \nu (j, \diff x) <\infty$} for all $j \in E$.
\end{remark}
{Let
\[
\zeta := \min \{ t > 0: H(t) \neq H(t-) \}
\] be the epoch of the first regime switch, which is exponentially distributed with parameter $q_i  := \sum_{j \neq i} q_{ij}$ under $\bE_{(x, i)}$. 
}

We define an operator $\Gamma$ applied to $f: [0,\infty) \times E \to \mathbb{R}$ defined by
	\begin{equation}\label{Supremum_Operator}
		(\Gamma f)(x,i) := \sup_{\pi \in \Pi} {\bE}_{(x,i)}\left[ \int_{[0,\zeta)} e^{-\qq(i)t} \diff L^{\pi}(t) - \beta \int_{[0,\zeta)} e^{-\qq(i)t} \diff R^{\pi}(t) + e^{-\qq(i)\zeta} \widehat{f}(U^{\pi}(\zeta -),i) \right],
	\end{equation}
	where
		\begin{align}\label{Expectation_Transform}
		\widehat{f}(x,i) := \sum_{j \neq i} \frac{q_{ij}}{q_i} \int_{\R} \left[ \left( \beta(x+y)+ f(0,j) \right)\Ind_{\lbrace x +y \leq 0 \rbrace} + f(x+y,j) \Ind_{\lbrace x+y > 0 \rbrace} \right] \diff F_{ij}(y),
	\end{align}
	where {we recall} $F_{ij}$ is the distribution function of the random variables {$J_{ij}^{(n)}$}.
	{Here, the supremum is over admissible (single-regime) strategies $\Pi$ in the sense defined in Section \ref{subsection_single_regime_formulation}. That is, \eqref{Supremum_Operator} is the value function of a single-regime version as studied in previous sections (see \eqref{vf_def}) driven by $X^i$ with the terminal time given by the first jump time of $H$ and the final payoff given by $\widehat{f}$.}

As shown in, for example \cite{MMNP,NobPerYu}, the value function is a fixed point of the maximization operator $\Gamma$ defined in \eqref{Supremum_Operator}. 
The proof is omitted as it follows verbatim from the proof of Proposition 3.4 in \cite{NobPerYu} {(see also the proof of Proposition 3.1 in \cite{MMNP})}. 


\begin{proposition}\label{DPP}
For $x\in\bR$ and $i\in E$, we have
\begin{align}
{V(x, i)} = &
\sup_{\pi \in \Pi}
\bE_{(x, i)} 
\bigg{[}\int_{[0, {\zeta})} e^{- {{\qq(i)t} }} \diff {L^\pi(t)} 
-\beta\int_{[0, {\zeta})}  e^{-{{\qq(i)t} }} \diff {R^\pi(t)} 
+
e^{{-{\qq(i) \zeta} }}V(U^\pi (\zeta), H(\zeta))\bigg{]}. \label{dpp}
\end{align}
{In other words, $V$ is a fixed point of $V = \Gamma V$.}
\end{proposition}




	\subsection{Optimal Strategies with Regime Switching}
	
	The {Markov-modulated} double barrier strategy \linebreak{$\pi^{ {(0,{\bb})}}= \{  ({L^{(0,\bb)}(t), R^{(0,\bb)}(t))}  :  t \geq 0\}$ with a set of upper barriers} ${\bb} = ({\bb}(i))_{i \in E}$ is {a} strategy {that} pushes the process $X$ downward by paying dividends whenever the process $X$ attempts to upcross above $\bb(i)$ {when the current regime is} $i$, while {pushing it} upward by injecting capital whenever the process $X$ attempts to downcross below $0$.  The corresponding controlled process  $U^{(0, \bb)} := X -   L^{(0,\bb)} +  R^{(0,\bb)}$ can be constructed in an obvious way; it is nothing but the concatenation of the doubly-reflected \lev process of Section \ref{Sec103}. 
	We denote its expected NPV by
	\[
	V_{\bb} (x, i) := V_{\pi^{ {(0,{\bb})}}} (x, i) = \mathbb{E}_{(x,i)} \left[ \int_{[0,\infty)} e^{-\Lambda(t)}   \diff {L^{(0,\bb)}(t)}  -  \beta\int_{[0, \infty)} e^{-\Lambda(t)} \diff {R^{(0,\bb)}(t)} \right], \quad x \geq 0, \; i \in E.
	\]
	
	The main theorem is the following. 
	\begin{theorem}\label{main_regime_switch}
	Under Assumptions \ref{AAA1}, and \ref{AAA2}, there exists a Markov-modulated 
	double barrier strategy which is optimal {and achieves \eqref{vf_rs}.}	
	\end{theorem}

	We proceed with the proof of Theorem \ref{main_regime_switch} as follows. We introduce in \eqref{Recursion_Operator} the operator $T_{\bb}$, acting on the space $\mathcal{B}$ given in \eqref{space_b}.
	 We show that the expected NPV under the Markov-modulated double barrier strategy, 
	{$V_{\bb}$}, is a fixed point of the operator $T_{\bb}$ ({see} Proposition \ref{Prop_FixedPoint}) and that $T_{\bb}$ is a contraction mapping. In the next step we show that the value function $V$ belongs to the class $\mathcal{D}$ defined in \eqref{57}, {which} is the content of Proposition \ref{Prop_Supremum_FixedPoint}.
	This fact allows us to apply Theorem \ref{Thm401} to show that there exists {${\bb}^* = ({\bb}^*(i))_{i \in E}$} such that the right hand-side of \eqref{dpp} is given by the expected NPV under a double barrier strategy at $\bb^*(i)$ for each $i \in E$. Hence, an application of the dynamic programming principle (Proposition \ref{DPP}), gives that $\Gamma V=T_{\bb^\ast}V=V$ where $\Gamma$ is the operator defined in \eqref{Supremum_Operator}. 
	Finally using that both 
	{$V_{\bb^*}$} and $V$ are fixed points of the contraction mapping $T_{\bb^*}$
	implies that  $V=V_{\bb^\ast}$. 
	
	{\begin{remark}
	{In the proof of Theorem \ref{main_regime_switch} we use a fixed point argument, hence} we cannot obtain explicitly the optimal barrier of a Markov-modulated double barrier strategy as in Theorem \ref{Thm401}. However, we can obtain it by approximation procedure. For more details, see the proof of Proposition \ref{Prop_Supremum_FixedPoint}.  
	\end{remark}}

	First, we show that the NPV, {$V_{\bb}$}
	 under a double barrier strategy solves a {recursive functional equation} {for any ${\bb} = ({\bb}(i))_{i \in E}$}.
	
	We consider the space of functions
	\begin{equation}\label{space_b}
	\B := \lbrace f : f(\cdot,i) \in C([0, \infty)), x \mapsto f(x,i)/\left(1+|x|\right) \text{ is bounded for all $i \in E$} \rbrace,
	\end{equation}
	endowed with the norm $\displaystyle\| f \| := \max_{i \in E} \sup_{x \geq 0} \frac{|f(x,i)|}{\left(1+|x|\right)}$.  
	{
		\begin{remark}
		Unlike the case of absolutely continuous dividend payments where the value function is bounded (see \cite[Section 5]{NobPerYu}), here the value function $V$ has linear growth. Therefore we introduce the norm $\| \cdot\|$ since  $\| V\|<\infty$.
	\end{remark}
}


%

\begin{remark}\label{bound_tilde_f}
Note that, for $(x,i) \in [0,\infty) \times E$ and $y\in\R$,
		\begin{align*}
	\frac{ \beta(x+y)+ f(0,j) }{1+|x|}\Ind_{\lbrace x +y \leq 0 \rbrace} + \frac{f(x+y,j)}{1+|x|} \Ind_{\lbrace x+y > 0 \rbrace}&\leq \beta
	{\frac { (|x|+|y|)} {1+|x|}}
	+ \frac{ |f((x+y)\lor 0,j)|}{1+|(x+y)\lor 0|}\frac{1+|(x+y)\lor 0|}{1+|x|} \notag\\
	&\leq \beta\left( 1+|y| \right)+ \| f \| (1 + |y|).
		\end{align*}
	{Under Assumption \ref{AAA2}}, for $(x,i) \in [0,\infty) \times E$, we have
	\[
	\int_{\R}\left[\beta\left( 1+|y| \right)+ \| f \| (1 + |y|)\right] \diff F_{ij}(y)= (\beta+\| f \|)\left( 1+\E\left[|J_{ij}^{(1)} |\right] \right)<\infty. 
	\]
Hence, \eqref{Expectation_Transform} and dominated convergence imply that the mapping $x\mapsto\frac{ \widehat{f}(x,i) }{1 + |x|}$ is continuous {on $[0,\infty)$} and therefore, so is the mapping $x\mapsto \widehat{f}(x,i)$. Additionally,
	\begin{align}
	\frac{ |\widehat{f}(x,i)| }{1 + |x|} \leq 
	\sum_{j \neq i} \frac{q_{ij}}{q_i}  \left(1+\E\left[|J_{ij} |\right]
	\right) \left( \beta +\|f\|\right).	\label{56}
	\end{align}
	Hence, if $f \in \B$, we have that $\widehat{f} \in \B$ as well. 
	\end{remark}
	Given ${\bb} = ({\bb}(i))_{i \in E}$ we define the following operator acting on $\B$ 
	\begin{align}
			(T_{{\bb}} f)&(x,i) := \bE_{(x,i)}\left[ \int_{[0,\zeta)} e^{-\qq(i)t} \diff L_i^{ {(0,{\bb}(i))}}(t)  - \beta \int_{[0,\zeta)} e^{-\qq(i)t} \diff R_i^{ {(0,{\bb}(i))}}(t) +  e^{-\qq(i)\zeta} \widehat{f}(U_i^{ {(0,{\bb}(i))}}(\zeta-),i) \right]\notag\\
			 &=\bE_{(x,i)}\left[ \int_{[0,\infty)} e^{-{\boldsymbol \alpha}(i)t} \diff L_i^{ {(0,{\bb}(i))}}(t)  - \beta  \int_{[0,\infty)} e^{-{\boldsymbol \alpha}(i)t} \diff R_i^{ {(0,{\bb}(i))}}(t) + q_i\int_0^{\infty}e^{-{\boldsymbol \alpha}(i)t}\widehat{f}(U_i^{ {(0,{\bb}(i))}}(t),i)\diff t \right], \label{Recursion_Operator}
		\end{align}
	where the processes $(U_i^{ {(0,{\bb}(i))}}, L_i^{ {(0,{\bb}(i))}}, R_i^{ {(0,{\bb}(i))}}: i \in E)$ are those under the double-barrier strategy with upper barrier ${\bb}(i)$ driven by $X^i$
	and
	\[
	\boldsymbol \alpha(i):=\qq(i)+q_i, \quad {i \in E.}
	\]

		{The proofs of the following two results are omitted as they follow verbatim from the proofs of Proposition 5.3  and Corollary 5.4 in \cite{NobPerYu}, respectively.
		}
	\begin{proposition}\label{Prop_FixedPoint} For ${\bb} = ({\bb}(i))_{i \in E}$, we have,  for $(x,i) \in [0,\infty) \times E$, 
		\[
		{V_{\bb}}(x,i) = (T_{{\bb}} 
		{V_{\bb}})
		(x,i). 
		\]
	\end{proposition}

Define $\| h \|_{\infty} := \max_{i \in E} \sup_{x \geq 0} | h(x,i)|$ for any $h: [0,\infty) \times E \to \R$.

	\begin{corollary}\label{cor_conv}
	For any $f,g \in \B$ such that $\| f-g \|_{\infty} < \infty$, we have $$\|T_{{\bb}}f-T_{{\bb}}g\|_{\infty} < K \|f-g\|_{\infty},$$ where 
	\begin{align}
	K:= \max_{i \in {E}} {\bE}_{(0,i)}\left[ e^{-\qq(i)\zeta} \right] \in (0,1). \label{def_K}
	\end{align}
	\end{corollary}
\subsection{Verification of barrier strategies.}  \label{subsection_verfication_barrier_MAP}
	We define the space of functions 
	\begin{align}
		\mathcal{D}:= \lbrace f \in \B : \widehat{f}(\cdot,i) &\text{ is concave and satisfies that $(\widehat{f})'_+(0,i)\leq \beta$ }\notag\\&\hspace{4cm}\text{and $(\widehat{f})'_+(\infty,i) := \lim_{x \to \infty}(\widehat{f})'_+(x,i)\in[0,1]$ for  all $i \in E$} \rbrace. \label{57}
	\end{align}
The proof {of the} next result follows closely  that of Proposition 5.5 in \cite{NobPerYu}, {and} so we omit it.
	\begin{proposition}\label{gammainD}
		{If} $f \in \B$ {is} such that,  for all $i \in E$, $f(\cdot,i) \in \mathcal{C}^1([0,\infty))$, 
		is concave and nondecreasing, and satisfies that $f'_{+}(\cdot,i)\leq \beta$ and $f'_{+}(\infty,i)\leq 1$, then $f \in \DD$. 
	\end{proposition}

	\begin{remark}\label{Remark_Optimal_Barrier}
		{If} $f\in \DD$, 
		then \eqref{Supremum_Operator} has the same form as the 
		single-regime problem with exponential horizon with terminal payoff $\widehat{f}$ studied in Section \ref{Sec02}. {Then, from the results of Subsection \ref{Subsec_Aux_Candidate_Threshold}, the supremum of \eqref{Supremum_Operator} is attained by a double reflection strategy at the barrier for some upper barrier, say $b_i^*$  for each $i \in E$. By setting $\bb^f = (\bb^f(i))_{i\in E}$ where $\bb^f(i) = b_i^*$ for each $i\in E$,}
		 we get $\Gamma f =  {T_{\bb^f}} f$. 
%
%
		
		{In addition, } from the verification results of Subsection \ref{Subsec_Aux_Verification} it follows that $(\Gamma f)(\cdot,i) \in C^{(1)}_\text{line}$ for all $i \in E$ and satisfies Assumption \ref{AA3}, hence 
		by Proposition \ref{gammainD}, $\Gamma f \in \DD$.	
	\end{remark}
The following result follows the same line of reasoning as in Proposition 5.7 in Noba et al. \cite{NobPerYu}.
	\begin{proposition}\label{Prop_Supremum_FixedPoint}
		Suppose $v_0^-, v_0^+ \in \DD$ are such that $v^-_0 \leq V \leq v^+_0$ and $\| v^+_0-v^-_0 \|_{\infty}<\infty $. We define recursively $v_n^- := \Gamma v^-_{n-1}$ and $v^+_n := \Gamma v^+_{n-1}$ for $n \geq 1$. Then, we have $v^-_n \leq V \leq v^+_n$ for all $n \geq 1$. Moreover, we have
		\[
		V= \lim_{n \rightarrow \infty} v^-_n = \lim_{n \rightarrow \infty} v^+_n,
		\]
		where the convergence is in the ${\|\cdot\|}_{\infty}$-norm. In particular, $V \in \DD$.
	\end{proposition}
	\begin{proof}
		First, from the definition of $\Gamma$ we have that if $v^-_{n-1} \leq V \leq v^+_{n-1}$ for some $n \geq 1$, then
		\begin{equation}\label{ine_V}
		v^-_n \leq \Gamma V \leq v^+_n.
		\end{equation}
		Moreover, due to Proposition \ref{DPP} we have
		$V = \Gamma V$, hence the first claim follows.
		
		Following Remark \ref{Remark_Optimal_Barrier}, for any $f \in \DD$ there exists a vector ${\bb}^f$ such that $\Gamma f = \sup_{\bb} T_{\bb} f = T_{{\bb}^f} f$.
		Hence, for $f,g \in \DD$ {such that $\| f-g \|_{\infty} < \infty$,}
		\[
		\| \Gamma f - \Gamma g \|_{\infty}  {=\| T_{\bb^f} f - T_{\bb^g} g \|_{\infty}} \leq \sup_{\bb} \| T_{\bb} f - T_{\bb} g \|_{\infty} \leq K \| f-g \|_{\infty},
		\]
		{with $K$ as in \eqref{def_K}.}
{
To see how the first inequality holds, for any $(x,i) \in [0, \infty) \times E$, if $(T_{\bb^f} f) (x,i) \geq  (T_{\bb^g} g) (x,i) $,
\begin{align*}
0 \leq (T_{\bb^f} f) (x,i)- (T_{\bb^g} g) (x,i) \leq (T_{\bb^f} f) (x,i)- (T_{\bb^f} g) (x,i) \leq  \sup_{\bb} \| T_{\bb} f - T_{\bb} g \|_{\infty};
\end{align*}
similarly, if $(T_{\bb^f} f) (x,i) \leq  (T_{\bb^g} g) (x,i) $, we obtain the same bound by symmetry.
Taking the supremum of their absolute values, we obtain the inequality.}		
		{
		Thus, {by iterating} the previous identity and the definition of $v_n^{+}$ and $v_n^{-}$, we obtain
		\[
		\| v_n^{+} - v_n^{-} \|_{\infty} \leq K^n \| v_0^{+} - v_0^{-} \|_{\infty} \quad \text{for } n \in \N.
		\]
		Since $K \in (0,1)$ it follows that $\| v_n^{+} - v_n^{-} \|_{\infty} \rightarrow 0$ as $n \rightarrow \infty$. Hence, using \eqref{ine_V} {together with Proposition \ref{DPP} gives}
		\begin{align}
		\lim_{n\to\infty}v_n^+=\lim_{n\to\infty}v_n^-=\Gamma V=V, \label{58}
		\end{align}
		{in the $\|\cdot\|_{\infty}$-norm.}
	}

		
		Following Remark \ref{Remark_Optimal_Barrier} we have that the functions $v_n^+, \, v^-_n$ belong to $\DD$ for all $n \geq 1$. 
		Hence, by \eqref{ine_V} we have 
				\begin{align*}
			&\|\widehat{V}-\widehat{v_n^{\pm}}\|_{\infty} \\
			&\leq \sup_{x\in\bR, i\in E} \sum_{j \neq i}  \frac{q_{ij}}{q_i} \Bigg( \int_{(-x,\infty)}  |V(x+y, j)-v_n^{\pm} (x+y,j)|  \diff F_{ij}(y) 
			+|V(0,j)-{v_n^\pm}(0,j) | F_{ij}(-x)\Bigg) \\
			&\leq \sup_{x\in\bR, i\in E} \sum_{j \neq i}  \frac{q_{ij}}{q_i} \int_{\R} \|{V}-v_n^{\pm}\|_{\infty} \diff F_{ij}(y) 
			=  \|{V}-v_n^{\pm}\|_{\infty}.
		\end{align*}
		This together with \eqref{58} shows
		\begin{align}
		\lim_{n\uparrow\infty}\|\widehat{V}-\widehat{v_n^\pm}\|_{\infty} =0. \label{59}
		\end{align}
		The uniform convergence of $(\widehat{v_n^\pm})_{n\in\mathbb{N}}$ to $\widehat{V}$ as in \eqref{59}, and the fact that $(v_n^{\pm})_{n\in\mathbb{N}}\subset\mathcal{D}$ implies that the concavity, as well as the properties of the right-hand derivative are inherited by $\widehat{V}$. Thus, $V\in\DD$.
	\end{proof}
	We will provide two auxiliary results that will be used in the proof of Theorem \ref{main_regime_switch}. The first result guarantees the existence of functions $v_0^-, v_0^+ \in \DD$ that satisfy the conditions of Proposition \ref{Prop_Supremum_FixedPoint}, and its proof is deferred to Appendix \ref{appendix_lemma_vv}.
	\begin{lemma}\label{proof_lemma_vv}
		There exist $V_{-}, V_{+} \in \DD$ such that $\|V_--V_+\|_{\infty}<\infty$ and 
		\[
		V_{-}(x,i) \leq V(x,i) \leq V_{+}(x,i), \quad (x,i) \in [0,\infty) \times E.
		\]
	\end{lemma}
	
	We now provide the second auxiliary result and defer its proof to Appendix \ref{Sec00A?}.
	\begin{lemma}\label{inf_norm_finite}
		For any $\bb=(\bb(i))_{i\in E}$, 
		we have that $\|V-
		{V_{\bb}}
		\|_{\infty}<\infty$.
\end{lemma}
Using Lemma \ref{proof_lemma_vv} together with Proposition \ref{Prop_Supremum_FixedPoint}, we provide an iterative construction of the value function as follows: 
Initialize by $n=0$ and $v=v_0$ for some $v_0\in\DD$ {that satisfies the conditions of Proposition \ref{Prop_Supremum_FixedPoint}}, we can operate the iteration scheme as in Section 4.1 in \cite{JP2012} by 
	\begin{itemize}
		\item[(1)] Find the barriers $\bb^v=(\bb^v(i);i\in E)$ 
		as in Remark \ref{Remark_Optimal_Barrier};
		\item[(2)] Set $T_{\bb^v}v\to v$
		and return to step (1).
	\end{itemize}
	By Proposition \ref{Prop_Supremum_FixedPoint}, this algorithm produces a sequence that converges to the true value function $V$.
	{
	\subsection{Proof of Theorem \ref{main_regime_switch}}
	Proposition \ref{Prop_Supremum_FixedPoint} gives that $V\in\DD$. 
	Hence Proposition \ref{DPP}, together with Remark \ref{Remark_Optimal_Barrier}, implies that there exists $\bb^*:=(\bb^*(i))_{i\in E}$ such that {$V=\Gamma V=T_{\bb^*}V$}. 
	Finally, an application of Proposition \ref{Prop_FixedPoint} together with Corollary \ref{cor_conv} concludes that
	\begin{align}\label{fix_point_final}
		\|V- {V_{\bb^*}}
		\|_{\infty} = \|T_{\bb^*}^nV-T_{\bb^*}^n
		{V_{\bb^*}}
		\|_{\infty}\leq  
	 K^n \| V-
	 {V_{\bb^*}}
	 \|_{\infty},\qquad\text{$n\in\mathbb{N}$.}
	\end{align}
	{By Lemma \ref{inf_norm_finite} we have that $\|V-
	{V_{\bb^*}}
	\|_{\infty}<\infty$. This together with \eqref{fix_point_final} implies that $V(x, i)={
	{V_{\bb^*}}
	}(x,i)$ for $(x,i)\in[0,\infty)\times E$.}
	\qed


	\subsection{Numerical experiments}
	
	We conclude this section with numerical experiments to confirm the convergence using a simple example. We consider the two-regime case ($E = \{ 0, 1\}$)  with
	\begin{align*}
	\diff X^i(t) = \mu_i \diff t + \sigma \diff W(t) + \diff Z^+(t) - \diff Z^-(t), \quad t \geq 0, \quad i \in E.
	\end{align*}
Here, $W = (W(t); t \geq 0)$ is a standard Brownian motion, $Z^+ = (Z^+(t) := \sum^{N^+(t)}_{j=1} Y_j^+; t \geq 0)$ and $Z^- = (Z^-(t) := \sum^{N^-(t)}_{j=1} Y_j^-; t \geq 0)$ are compound Poisson processes where $N^+$ (resp.\ $N^-$) is a Poisson process with rate $\lambda^+$ (resp.\ $\lambda^-$) and $Y^+$ (resp.\ $Y^-$) is an i.i.d.\ sequence of random variables. These processes are assumed mutually independent. 

For this experiment, we set $\mu_0 = 1.5$, $\mu_1 = 1.1$, $\sigma = 0.2$, $\lambda^+ = 0.8$, $\lambda^- = 0.2$ and $Y^+$ (resp.\ $Y^-$) is a sequence of (1-parameter) Weibull random variables with shape parameter $2$ (resp.\ the absolute values of  standard normal random variables). We assume $J_{ij}^{(n)} = 0$, $i \neq j$ and $n \geq 1$, and $q_{01} = q_{10} = 0.1$ for simplicity. For the state-dependent discount factor, we set $\qq(0) = 0.05$ and $\qq(1) = 0.075$.
	
We employ Monte-Carlo simulation. We first generate $M = 10,000$ number of paths of $(W, Z^+, Z^-)$ as well as the Markov chain $H$. With these, we obtain $M$ realizations of the (uncontrolled) Markov additive process $X$ when started at $H(0)=0,1$. These i.i.d.\ samples are used repeatedly to approximate expectations, instead of generating new sets at each iteration, to enhance the computational efficiency. We use the classical Euler-method with a truncated horizon $50$ and a time step of $\Delta t = 0.05$.

To confirm  the convergence of the iteration scheme outlined in (1)-(2) at the end of Section \ref{subsection_verfication_barrier_MAP}, we first set $\bb_0 = (\bb_0(0), \bb_0(1)) = (0.5,0.5)$ and recursively obtain  $\bb_{n+1} = (\bb_{n+1}(0), \bb_{n+1}(1))$ from $\bb_{n}$ via $v_n := T_{\bb^{v_{n-1}}} v_{n-1}$ for each $n \geq 1$. More precisely, the values of $\bb_{n+1}$ are obtained by solving \eqref{Supremum_Operator} for $f = V_{\bb_n}$ (where we already know that an optimal strategy is of barrier type). Table \ref{table_convergence} shows the values of $\bb_{n}$, confirming the convergence results obtained in the previous subsection. In our experiment, the convergence is fast and our algorithm stops after $5$ iterations.
\begin{table}
\begin{center}
\begin{tabular}{ |c|c|c| } 
 \hline
 $n$ & $\bb_{n}(0)$ & $\bb_{n}(1)$ \\ 
 \hline
 0 & 0.5 & 0.5 \\
 1 &1.0729839785687219 &0.8454152772983684 \\
 2 & 1.0724932895265710 & 0.8471510786024958 \\ 
3 & 1.0724622610024720 & 0.8471511240717061 \\ 
4 & 1.0724622494726461 & 0.8471625181626047 \\
5 & 1.0724622024111183 & 0.8471626732511570 \\
 \hline
\end{tabular}
\end{center}
\caption{The values of $(\bb_{n}; n = 0, 1, \ldots, 5)$ obtained when $\bb_0= (0.5,0.5)$.
} \label{table_convergence}
\end{table}

 In order to confirm that the obtained barriers $(\bb^*(0), \bb^*(1))  := (\bb_5(0), \bb_5(1))$ are indeed precise approximation of the optimal barriers, we verify that $\bb^*(0)$ is the optimal barrier when in state $0$ and that $\bb^*(1)$ is the optimal barrier when in state $1$. To this end, we analyze the  net present value of dividends $V_{\bb} = V_{(\bb(0),\bb(1))}$ under barrier strategies  with barriers $\bb = (\bb(0), \bb(1))$, considering $V_{\bb}$ as a function of $\bb(0)$ with $\bb(1)$ fixed at $\bb^*(1)$ and as a function of $\bb(1)$ with $\bb(0)$ fixed at $\bb^*(0)$. These analyses are shown in Figure \ref{plot_value}, focusing on the case where the starting value is $x=0$.
  These figures demonstrate that our obtained barriers $\bb^*(0)$ and $\bb^*(1)$ are indeed maximizers of
$\bb(0) \mapsto V_{(\bb(0), \bb^*(1)) }(0)$ and $\bb(1) \mapsto V_{(\bb^*(0), \bb(1)) }(0)$, respectively. 
This confirms the effectiveness of our Theorem \ref{main_regime_switch}-based algorithm.
    


\begin{figure}[htbp]
\begin{center}
\begin{minipage}{1.0\textwidth}
\centering
\begin{tabular}{cc}
 \includegraphics[scale=0.6]{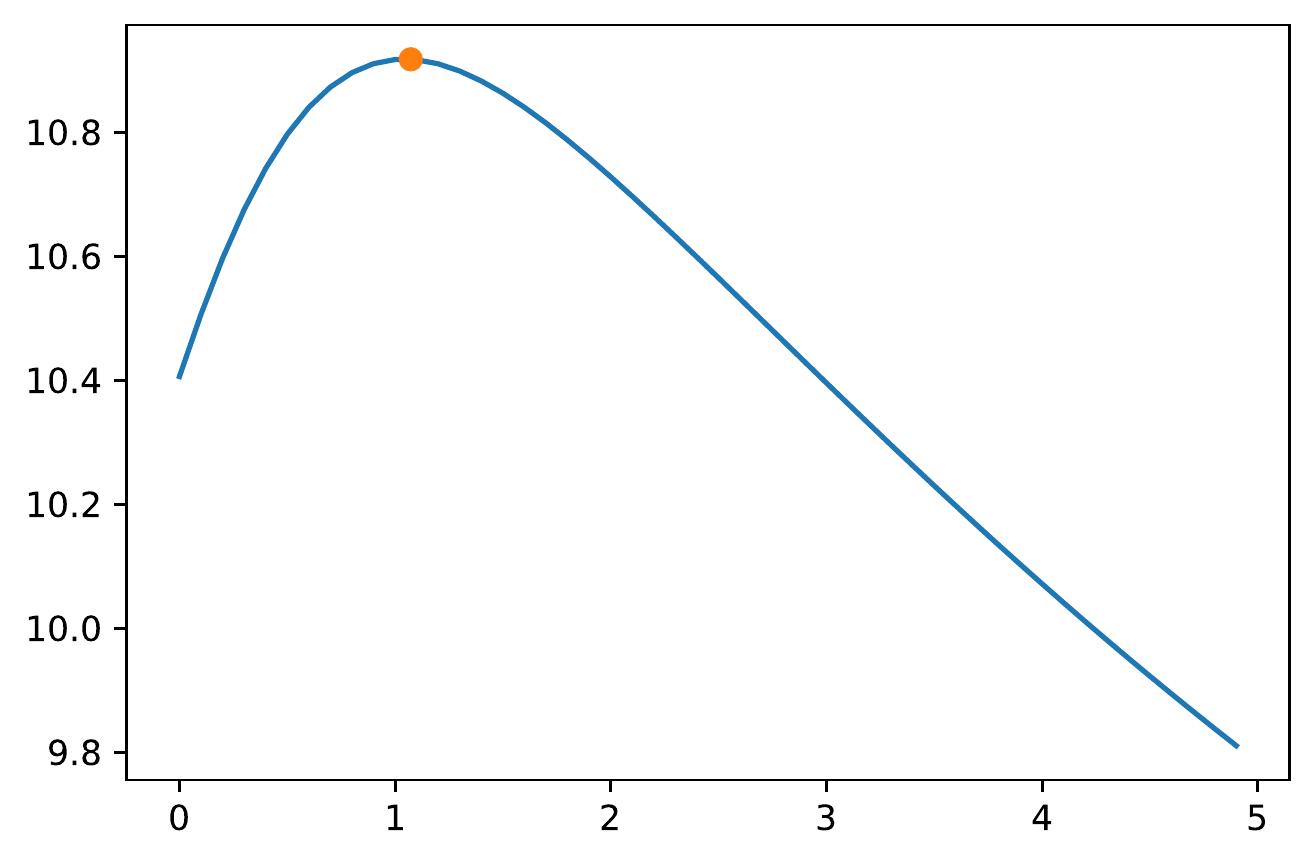} & \includegraphics[scale=0.6]{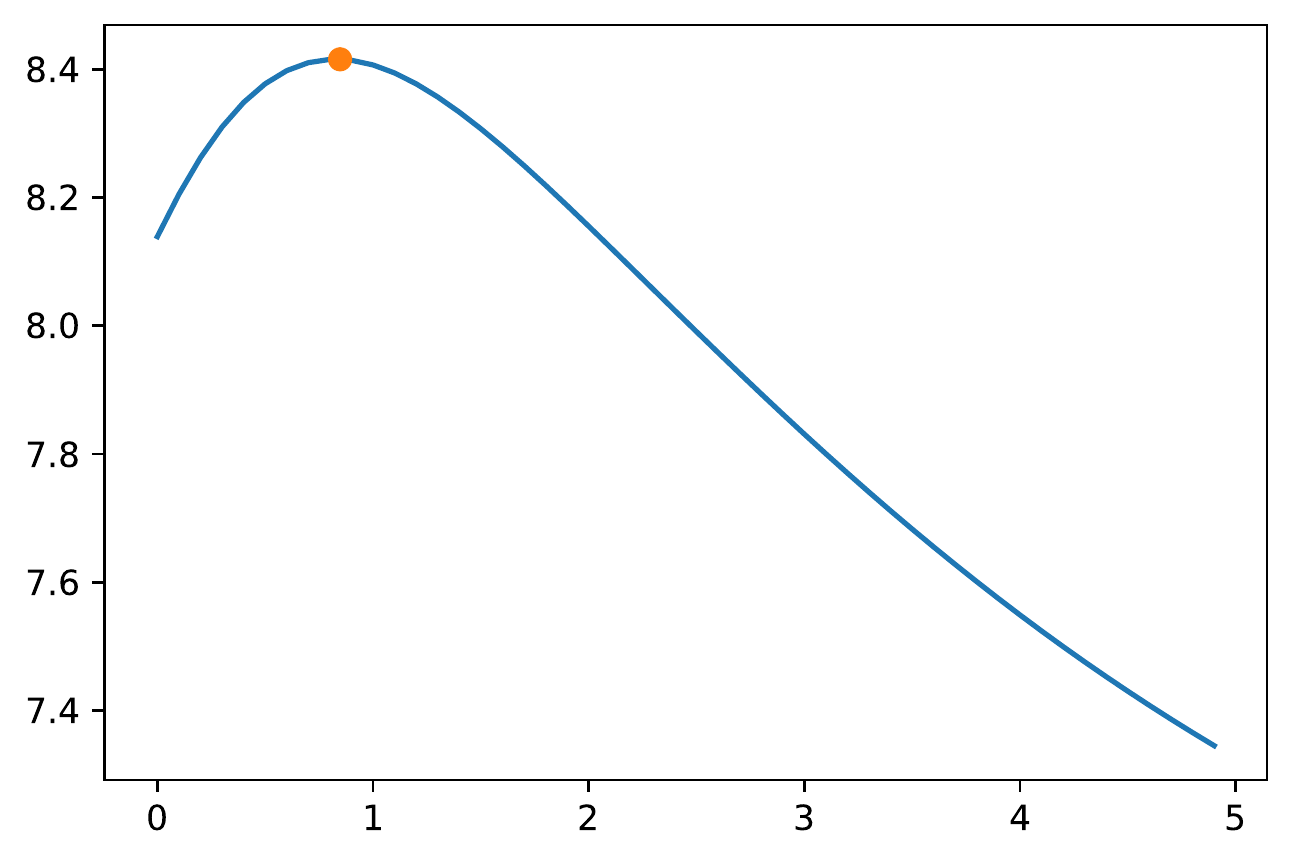}   \\
$\bb(0) \mapsto V_{(\bb(0), \bb^*(1)) }(0)$  & $\bb(1) \mapsto V_{(\bb^*(0), \bb(1)) }(0)$
 \end{tabular}
\end{minipage}
\caption{Optimality of $\bb^* = (\bb^*(0), \bb^*(1))$, obtained by the recursive algorithm. The left plot shows the values of $V_{\bb}(0)$ as a function of $\bb(0)$ for fixed $\bb(1) = \bb^*(1)$, while the right plot shows $V_{\bb}(0)$ as a function of $\bb(1)$ for fixed $\bb(0) = \bb^*(0)$. The values for the case $\bb(0) = \bb^*(0)$ and $\bb(1) = \bb^*(1)$ are indicated by circles.
} \label{plot_value}
\end{center}
\end{figure}

	\appendix
	
	\section{Proofs}\label{section_proofs}
	
	
	\subsection{Proof of Lemma \ref{Lem402}} \label{proof_Lem402}
		{(i)}  Fix $b\geq0$. 
We first prove
		\begin{align}
			\lim_{\delta\downarrow0}\kappa^{0,-}_{-b-\delta}=\kappa^{0,-}_{-b} \quad \textrm{a.s.\ on } \{ \kappa^{0,-}_{-b} < \infty\}. \label{proof_conv_kappa}
		\end{align}
Suppose $\kappa^{0,-}_{-b} < \infty$. Because $\delta \mapsto \kappa^{0,-}_{-b-\delta}$ is non-decreasing, 
$\lim_{\delta \downarrow 0}\kappa^{0,-}_{-b-\delta}$ exists a.s.; for \eqref{proof_conv_kappa} it suffices to show it coincides with $\kappa^{0,-}_{-b}$. To derive a contradiction, suppose otherwise, i.e., $\bar{\epsilon}:= \lim_{\delta \downarrow 0}\kappa^{0,-}_{-b-\delta} -  \kappa^{0,-}_{-b} > 0$, which implies, thanks to the monotonicity of $\delta \mapsto \kappa^{0,-}_{-b-\delta}$,
\begin{align}
\kappa^{0,-}_{-b-\delta}  > \kappa^{0,-}_{-b} + \bar{\epsilon}, \quad \delta > 0. \label{eq_cont}
\end{align}

Now, take $\underline{\epsilon} := \bar{\epsilon}/2 > 0$.
From the definition of $\kappa^{0,-}_{-b}$, 
there exists $\tilde{\delta}>0$ 
such that 
			$\inf_{s\in[0, \kappa^{0,-}_{-b}+\underline{\epsilon}]} {Y^0(s)}< -b-\tilde{\delta}$,
		which implies that  $\kappa^{0,-}_{-b}+\overline{\epsilon} > \kappa^{0,-}_{-b}+\underline{\epsilon} \geq  \kappa^{0, -}_{-b-\tilde{\delta}} \geq \kappa^{0,-}_{-b}$. This contradicts with \eqref{eq_cont}. Hence, \eqref{proof_conv_kappa} holds.

		{By \eqref{proof_conv_kappa} and} since $w_+^\prime$ is right-continuous, we have {by dominated convergence}
		\begin{align}
			{\lim_{\varepsilon\downarrow 0}\left( \beta-g(b+\varepsilon)\right)=}\lim_{\varepsilon\downarrow0}
			&\E_0 \left[ \int_0^{\kappa^{0,-}_{-(b+\varepsilon)}} e^{-\alpha t} l( {Y^0(t)}+b+\varepsilon) \diff t \right]\\
			&=\E_0 \left[ \int_0^{\kappa^{0,-}_{-b}} e^{-\alpha t}l ( {Y^0(t)}+b) \diff t \right]= {\beta-} g(b),
		\end{align}
showing the right-continuity.
	\par
	{(ii)} 	By the quasi-left-continuty of {the} L\'evy process (see, {e.g.,} Lemma 3.2 in \cite{Kyp2014}), we obtain
		\begin{align}
			 {Y^{0}(\lim_{\varepsilon\downarrow 0}\kappa^{0,-}_{-(b-\varepsilon)})}=\lim_{\varepsilon\downarrow 0} {Y^{0}(\kappa^{0,-}_{-(b-\varepsilon)})}\leq -b.
		\end{align}
		If the point $0$ is regular for $(-\infty,0)$ then 
		$Y^{0}$ hits $(-\infty , -b)$ immediately after time $\lim_{\varepsilon\downarrow 0}\kappa^{0,-}_{-(b-\varepsilon)}$.
		Therefore, $\lim_{\varepsilon\downarrow 0}\kappa^{0,-}_{-(b-\varepsilon)}=\kappa^{0,-}_{-b}$. 
	 	Since $\lim_{\varepsilon\downarrow0}w_+^\prime (x-\varepsilon)=w_-^\prime(x)$ for $x\in\bR$ by \cite[Proposition 3.1 in the Appendix]{RevYor1999},
		\begin{multline}
			{\lim_{\varepsilon\downarrow0}\left(\beta-g(b-\varepsilon)\right)=}\lim_{\varepsilon\downarrow0}
			\E_0 \left[ \int_0^{\kappa^{0,-}_{-(b-\varepsilon)}} e^{-\alpha t} l( {Y^0(t)}+b-\varepsilon) \diff t \right]\\
			=\E_0 \left[ \int_0^{\infty}\lim_{\varepsilon\downarrow0}\Ind_{\{t<\kappa^{0,-}_{-(b-\varepsilon)}\}} e^{-\alpha t} l ( {Y^0(t)}+b-\varepsilon) \diff t \right]
			=\E_0 \left[ \int_0^{\kappa^{0,-}_{-b}} e^{-\alpha t} l_- ( {Y^0(t)}+b))\diff t  \right],
					\end{multline}
					{with $l_-(y) := \beta \alpha -rw'_{-} (y)$.}
	It is shown after identity (3.10) in \cite{NobYam2020} that the potential measure of $Y^0$ has no mass when $X$ has unbounded variation paths. Similarly,  for the case where $X$ has bounded variation paths, we follow the same computation leading to the equation at the bottom of \cite[p.1217]{NobYam2020}. Then, using the fact that $Y^0$ reaches $0$ a finite number of times on a compact interval when $0$ is regular for $(-\infty,0)$, as stated in \cite[p.158]{Kyp2014}, we conclude that $Y^0$ has no mass on $(-\infty, 0]$. 
Since $w_-^\prime \neq w_+^\prime$ at most countably {many points} by {\cite[Proposition 3.1, Appendix]{RevYor1999}}, {$l_-(\cdot)$ can be replaced with $l(\cdot)$ in the last expectation, showing  $\lim_{\varepsilon\downarrow0}g(b-\varepsilon) =g(b)$ when  $0$ is regular for $(-\infty,0)$, as desired.} 

\subsection{Proof of Lemma \ref{lemma_example_condition}}\label{Sec00B} 
{In this section, we provide} an example of {a L\'evy process} $X$ which has unbounded variation paths and satisfies Assumption \ref{Ass306}. 
	\par
	Assume that $X$ satisfies {that} $\nu (0, \infty) <\infty $. 
	Then {we can write the process $X$} as follows: 
	\begin{align}
	X(t) ={X^{SN}(t)} + \sum_{n=1}^{N^\nu(t)} J_n
	\end{align} 
	where ${X^{SN}(t)}=\{X^{SN}(t) : t\geq 0\}$ is a spectrally negative L\'evy process with unbounded variation paths, $N^\nu=\{N^\nu(t) : t\geq 0\}$ is {an} independent Poisson process with intensity $\nu(0, \infty)$ and ${\{ J_n \}}_{ n \in \N}$ is a sequence of i.i.d.\ random variables {with} distribution $\nu(\cdot \cap (0, \infty))/\nu(0, \infty)$. 
	{For $\theta\geq 0$,} let $W^{(\theta)}$ {denote the} $\theta$-scale function of ${X^{SN}}$. 
	For the definition of the scale functions, see, e.g., \cite[Section 8]{Kyp2014}.
	{For $b>0$,} let ${Y^{SN, b}}=\{{Y^{SN, b}(t)} : t\geq 0\}$ with $b>0$ the {associated} reflected L\'evy process of $X^{SN}$ at $b$, i.e. 
	${Y^{SN, b}(t)}= {X^{SN}(t)} - \sup_{s \in [0, t]} \left\{({X^{SN}(s)} -b) \land 0\right\}$. 
	In addition, we define ${\kappa^{SN, b}_0} := \inf\{t> 0: {Y^{SN, b}(t)}< 0\}$ and 
	\begin{align}
	{H^{SN, (b, \theta)}_f}(x):=\E_x \left[\int_0^{{\kappa^{SN, b}_0} }e^{-\theta t} f({Y^{SN, b}(t))} \diff t \right], \quad \theta>0 , x\in\bR.  
	\end{align} 
	From \cite[Theorem 8.11]{Kyp2014} and by our assumption that $X$ is of unbounded variation, we have, for $\theta \geq 0$ and $x\in (0, b)$ and {a} non-negative measurable function $f$, 
	\begin{align}
	\E_x \left[\int_0^{{\kappa^{SN, b}_0} }e^{-\theta t} f({Y^{SN, b}(t))} \diff t \right]
	=\int_0^b f(y) \left( W^{(\theta)}(x)\frac{W^{(\theta)\prime}(b-y)}{W^{(\theta)\prime}(b)} -W^{(\theta)}(x-y)\right)\diff y. 
	\end{align}
	Thus, we have  
	\begin{align}
	H^{(b, \alpha)}_f(x)
	=&{H^{SN, (b, \alpha)}_f}(x) -\E_x \left[e^{-\alpha T^{N^\nu}}{H^{SN, (b, \alpha)}_f} (Y^{SN, b}_{T^{N^\nu}}) ; T^{N^\nu}< \kappa^{SN, b}_0\right]
	\\
	&+\E_x  \left[ e^{-\alpha T^{N^\nu}}{H^{(b, \alpha)}_f(Y^{SN, b}_{T^{N^\nu}}+J_1)} ; T^{N^\nu}< \kappa^{SN, b}_0\right]\\
	=&{H^{SN, (b, \alpha)}_f}(x)
	-\nu(0, \infty) {H^{SN, (b, \hat{\alpha})}_{H^{SN, (b, \alpha)}_f}}(x)
	+\int_{(0, \infty)} {H^{SN(b, \hat{\alpha})}_{H^{ (b, \alpha)}_f(\cdot + u) }}(x)\nu(\diff u), 
	\end{align}
	where $T^{N^\nu}$ is the first jump time of $N^\nu$ and $\hat{\alpha}=\alpha+\nu(0, \infty)$. 
	Since {the process $X^{SN}$ is of unbounded variation, it is known that $W^{(\theta)}(0+)=0$, hence} for a function $g\in C^1 (\bR)$, we have
	\begin{align}
	{H^{SN, (b, \theta)}_g} (x)=
	\int_0^x \int_0^b g(y) \left( W^{(\theta)\prime}(z)\frac{W^{(\theta)\prime}(b-y)}{W^{(\theta)\prime}(b)} -W^{(\theta)\prime}(z-y)\right)\diff y \diff z,
	\end{align}
	the function $H_f^{(b, \alpha)}$ has the density 
	\begin{align}
	H_f^{(b, \alpha),\prime}(z):=&
	\int_0^b f(y) \left( {W^{(\alpha)\prime}(z)\frac{W^{(\alpha)\prime}(b-y)}{W^{(\alpha)\prime}(b)} -W^{(\alpha)\prime}(z-y)}\right)\diff y\\
	&-\nu(0, \infty)\int_0^b {H^{SN, (b, \alpha)}_f}(y) \left({W^{(\hat{\alpha})\prime}(z)\frac{W^{(\hat{\alpha})\prime}(b-y)}{W^{(\hat{\alpha})\prime}(b)} -W^{(\hat{\alpha})\prime}(z-y)}\right)\diff y\\
	&+\int_0^b\int_{(0,\infty)} {H^{ (b, \alpha)}_f}(y+u)\nu(du) \left( {W^{({\hat{\alpha}})\prime}(z)\frac{W^{({\hat{\alpha}})\prime}(b-y)}{W^{({\hat{\alpha}})\prime}(b)} -W^{({\hat{\alpha}})\prime}(z-y)}\right)\diff y
	,\ z\in(0, b].
	\end{align}
	Since $f$ is bounded on $[0, b]$ and 
	\begin{align}
	\sup_{x \geq 0}|H^{(b, \alpha)}_f(x )|
	\lor \sup_{x \geq 0}|H^{{SN}, (b, \alpha)}_f(x )| \leq \frac{1}{\alpha} \sup_{x\in[0, b]}|f(x)|,
	\end{align}
	the density $H_f^{(b, \alpha),\prime}$ is bounded on any compact intervals in $(0, b]$. 
	In addition, it is not difficult to prove the continuity of $H_f^{(b, \alpha),\prime}$. 
\par
In the same way, we can prove that the case with $\nu(-\infty , 0)<\infty$ satisfies Assumption \ref{Ass306}.

\subsection{Proof of Lemma \ref{aux_1}}	\label{AppA03}
	First, we have
		\begin{align}
		g_0(b^*) = g(b^*) \leq 1,  \label{g_0}
		\end{align}
		where the inequality holds by  the definition of $b^\ast$ and since $g$ is right-continuous.
		
		On the other hand,
		\begin{align}
		1 \leq \lim_{b\uparrow b^\ast} g(b)
			=&\lim_{b\uparrow b^\ast}
			\left( \beta - \E_{b^*} \left[ \int_0^{\kappa^{b^*,-}_{b^*-b}} e^{-\alpha t} l ( {Y^{b^*}(t)}+b-b^*) \diff t  \right] \right) \\
%
			=&
			 \beta - \E_{b^*} \left[ \int_0^{K^{b^\ast, -}_0} e^{-\alpha t} l_- ( {Y^{b^*}(t)}) \diff t  \right] =  g_1(b^*),
			\label{37}
		\end{align}
		where the first inequality holds by the definition of $g$ (and by shifting vertifically the process $Y^{b^*}$) and the limit holds because $\kappa^{b^*,-}_{b^*-b} \xrightarrow{b \uparrow b^*} K^{b^{\ast}, -}_0$ a.s.

For $p \in [0,1]$, denoting $l_{\pm}(y):=\beta\alpha-rw'_{\pm}(y)$ for $y>0$,
					\begin{align*}
		g_p(b^*) &= \beta -  p  \E_{b^*} \left[ \int_0^{K^{b^*,-}_{0}} e^{-\alpha t} l_p ( Y^{b^*}(t)) \diff t   \right] - (1-p)  \E_{b^*} \left[ \int_0^{\kappa^{b^*,-}_{0}} e^{-\alpha t} l_p ( Y^{b^\ast}(t)) \diff t   \right] \\
		&= \beta -  p (1-p)  \E_{b^*} \left[ \int_0^{K^{b^*,-}_{0}} e^{-\alpha t} l_+(Y^{b^*}(t)) \diff t   \right] - (1-p)^2  \E_{b^*} \left[ \int_0^{\kappa^{b^*,-}_{0}} e^{-\alpha t} l_+ ( Y^{b^*}(t)) \diff t\right] \\
		 &-  p^2  \E_0 \left[ \int_0^{K^{b^*,-}_{0}} e^{-\alpha t} l_- ( Y^{b^*}(t)) \diff t   \right] - p (1-p) \E_{b^*} \left[ \int_0^{\kappa^{b^*,-}_{0}} e^{-\alpha t} l_- (Y^{b^*}(t)) \diff t   \right],
		\end{align*}
		which is a continuous function of $p$.
By this continuity together with the values of the end points as in  \eqref{g_0} and \eqref{37}, there exists $p^\ast \in[0, 1]$ satisfying \eqref{26} and the proof is complete. 

\subsection{Proof of Lemma \ref{proof_lemma_vv}}\label{appendix_lemma_vv}
	We define 
	\begin{align}
	V_-(x, i)&= x +
	V(0,i)
	\qquad\text{and} \label{60}\\ 
	V_+(x, i) &= {\bE}_{(x, i)} \left[\int_{[0, \infty)}e^{- \int_0^t \qq(H(s)) \diff s}\diff \overline{X}(t) \right]\\
 &= x+{\bE}_{(x, i)} \left[\int_{(0, \infty)}e^{- \int_0^t \qq(H(s)) \diff s}\diff \overline{X}(t) \right]\\
 &=x +{\bE}_{(0, i)} \left[\int_{[0, \infty)}e^{- \int_0^t \qq(H(s)) \diff s}\diff \overline{X}(t) \right],
	\end{align}
	where $\overline{X}(t):=\sup_{s\leq t}X(s)$. 
First, we have
	\begin{align}
		V_- (x , i)\leq V(x , i),\quad x\in\bR, i\in E. \label{40}
	\end{align}
	This is immediate by considering a subset of admissible strategies with an extra condition that the process must be moved instantaneously to $0$ at time $0$. The function $V_-$ is precisely the value function of this version with more constraints and thus is bounded from above by the value function $V$ of the original formulation. 
	

%
	
	Now, we prove that 
	\begin{align}
		V_+ (x , i)\geq V(x , i),\quad x\geq 0, i\in E. \label{41}
	\end{align}
	We fix $\pi\in\cA$. 
By Tonelli's theorem, it holds
 \begin{align}
 \int_{[0, \infty)}e^{- \int_0^t \qq(H(s)) \diff s} \diff \overline{X}(t)
 =&\int_{[0, \infty)} \int_t^\infty \qq(H(u)) e^{- \int_0^u \qq(H(s)) \diff s} \diff u \diff \overline{X}(t)\\
 =& \int_0^\infty \qq(H(u)) e^{- \int_0^u \qq(H(s)) \diff s} \int_{[0, u]}\diff \overline{X}(t) \diff u \\
 =& \int_0^\infty \qq(H(u)) e^{- \int_0^u \qq(H(s)) \diff s}  \overline{X}(u) \diff u .
 \end{align}
Since the same calculations can be made for both $L^\pi$ and $R^\pi$}, we have
	\begin{align}
		& \int_{[0, \infty)}e^{- \int_0^t \qq(H(s)) \diff s} \diff \overline{X}(t)
		- \left( \int_{[0, \infty)}e^{- \int_0^t \qq(H(s)) \diff s} \diff L^\pi(t) -\beta \int_{[0, \infty)}e^{- \int_0^t \qq(H(s)) \diff s} \diff R^\pi(t)  \right)\\
		&=\int_0^\infty  \qq(H(u)) e^{-\int_0^u \qq(H(s))\diff s} \left(\overline{X}(u)-L^\pi(u) + \beta R^\pi(u) \right)
		\diff u \geq 0,
	\end{align}
	where 
 in the inequality, we used condition \eqref{positiveness} for admissible strategies.
	By taking the expectation, we obtain \eqref{41}.

\subsection{Proof of {Lemma \ref{inf_norm_finite}}} \label{Sec00A?}	

By Lemma \ref{proof_lemma_vv} and the triangle inequality, it suffices to show $\|V_{\bb}-V_- \|_\infty<\infty$, where $V_-$ is the function defined in \eqref{60}. 

	By the definition of $\pi^{(0, \bb)}$, we have 
\begin{align}
V_{\bb}(x, i)=
V_{\bb}(\bb_M, i)+(x-\bb_M),\qquad &x\geq\bb_M, 
\end{align}
where $\bb_M:= \max_{i\in E}\bb(i)$, 
and is continuous by the same proof as \cite[Lemma 5.4]{Nob2019}. 
Thus, for $x\geq\bb_M$ and $i \in E$, 
\begin{align}
V_{\bb}(x, i) - V_-(x, i)=
V_{\bb}(\bb_M, i)+(x-\bb_M) - x - V(0,i) = V_{\bb}(\bb_M , i) -\bb_M -V(0, i), \label{62}
\end{align}
as desired.
%
%
From \eqref{62}, we have 
\begin{align*}
\|V_{\bb}-V_- \|_\infty &\leq \sup_{i\in E, x\geq \bb_M} |V_{\bb}(x, i)-V_- (x, i)| +\sup_{i\in E, x\in[0,\bb_M]} |V_{\bb}(x, i)-V_- (x, i)|\\
&\leq \sup_{i\in E, x\geq \bb_M} |V_{\bb}(x, i)-V_- (x, i)| +\sup_{i\in E, x\in[0,\bb_M]} |V_{\bb}(x, i)|
+\sup_{i\in E, x\in[0,\bb_M]} |V_- (x, i)|\\
&\leq \max_{i\in E} |V_{\bb}(\bb_M , i) -\bb_M -V(0, i)|+ \sup_{i\in E, x\in[0,\bb_M]} |V_{\bb}(x, i)|+(\max_{i\in E} |V (0, i)|+\bb_M)
<\infty. 
\end{align*}



	\section*{Acknowledgements}
	
	K. Noba was supported by JSPS KAKENHI grant no.\ JP21K13807. In addition, K.\ Noba stayed at Centro de Investigaci\'on en Matem\'aticas in Mexico as a JSPS Overseas Research Fellow and received support regarding the research
environment there. K. Noba is grateful for their support during his visit. K.\ Yamazaki was supported by JSPS KAKENHI grant no.\ JP20K03758, JP24K06844 and JP24H00328
and
 the start-up grant by the School of Mathematics and Physics of the University of Queensland. K.\ Noba and K.\ Yamazaki were supported by JSPS Open Partnership Joint Research Projects grant no. JPJSBP120209921 and JPJSBP120249936.

\end{document}